\documentclass[11pt,a4paper]{amsart}

\usepackage{latexsym}
\usepackage{amssymb,amsmath} 
\usepackage{graphicx}
\usepackage{url}
\usepackage[all]{xypic}

\theoremstyle{definition}
\newtheorem*{definition*}{Definition}

\theoremstyle{plain}
\newtheorem{theorem}{Theorem}[section]

\newtheorem{proposition}[theorem]{Proposition}

\newtheorem{lemma}[theorem]{Lemma}

\usepackage{amssymb,amsmath,amsopn, amsthm}

\renewcommand{\phi}{\varphi}
\renewcommand{\theta}{\vartheta}

\newtheorem{thmmain}{Theorem}

\newcommand{\F}{\mathbf{F}}

\newcommand{\Z}{\mathbf{Z}}
\newcommand{\N}{\mathbf{N}}

\newcommand{\Q}{\mathbf{Q}}

\renewcommand{\L}{\Lambda}

\DeclareMathOperator{\Cent}{Cent}
\DeclareMathOperator{\Stab}{Stab}
\DeclareMathOperator{\Sym}{Sym}
\DeclareMathOperator{\GL}{GL}
\DeclareMathOperator{\SL}{SL}
\DeclareMathOperator{\AGL}{AGL}
\DeclareMathOperator{\PGL}{PGL}
\DeclareMathOperator{\PSL}{PSL}
\DeclareMathOperator{\Con}{Con}
\DeclareMathOperator{\Aut}{Aut}
\DeclareMathOperator{\characteristic}{char}
\DeclareMathOperator{\Gal}{Gal}
\newcommand{\dotcupop}{\;\ensuremath{\mathaccent\cdot\cup}\; }

\newcommand{\join}{\vee}
\newcommand{\meet}{\wedge}
\newcommand{\partequiv}{\equiv}

\newcommand{\PGammaL}{\mathrm{P}\Gamma\mathrm{L}}

\newcounter{defnlistcnt}
\newenvironment{defnlist}%
    {\setcounter{defnlistcnt}{0}%
    \begin{list}{(\arabic{defnlistcnt})}{%
        \usecounter{defnlistcnt}%
        \setlength{\labelwidth}{18pt}%
        \setlength{\topsep}{1pt}%
        \setlength{\leftmargin}{30pt}%
        \setlength{\itemsep}{0pt}%
        \setlength{\itemindent}{0pt}}%
    }%
    {\end{list}}%

\newcounter{thmlistcnt}
\newenvironment{thmlist}%
    {\setcounter{thmlistcnt}{0}%
    \begin{list}{\emph{(\alph{thmlistcnt})}}{%
        \usecounter{thmlistcnt}%
        \setlength{\labelwidth}{18pt}%
        \setlength{\topsep}{1pt}%
        \setlength{\leftmargin}{30pt}%
        \setlength{\itemsep}{0pt}%
        \setlength{\itemindent}{0pt}}%
    }%
    {\end{list}}%

\newenvironment{bulletlist}{%
    \begin{list}{$\bullet$}{%
        \setlength{\labelwidth}{18pt}%
        \setlength{\topsep}{1pt}%
        \setlength{\leftmargin}{30pt}%
        \setlength{\itemsep}{0pt}%
        \setlength{\itemindent}{0pt}}%
    }%
    {\end{list}}%

\newcounter{numlistcnt}
\newenvironment{numlist}%
    {\setcounter{numlistcnt}{0}%
    \begin{list}{\emph{(\arabic{numlistcnt})}}{%
        \usecounter{numlistcnt}%
        \setlength{\labelwidth}{18pt}%
        \setlength{\topsep}{2pt}%
        \setlength{\leftmargin}{30pt}%
        \setlength{\itemsep}{2pt}%
        \setlength{\itemindent}{0pt}}%
    }%
    {\end{list}}%

\newcommand{\Prufer}{Pr{\"u}fer\ }
\renewcommand{\P}{\mathbf{P}}

\newcommand{\p}{\mathcal{P}}
\newcommand{\s}{\mathcal{S}}

\newcommand{\n}{n}

\title{Orbit coherence in permutation groups}

\author{John R. Britnell \and Mark Wildon}
\address{Heilbronn Institute for Mathematical Research, School of Mathematics,
University of Bristol, University Walk, Bristol, BS8 1TW}
\email{j.r.britnell@bristol.ac.uk}
\address{Department of Mathematics, Royal Holloway, University of London, Egham, Surrey TW20 0EX}
\email{mark.wildon@rhul.ac.uk}


\begin{document}


\bigskip

\begin{abstract}{
This paper introduces the notion of orbit coherence in a permutation group.
Let~$G$ be a group of permutations of a set~$\Omega$. Let~$\pi(G)$ be the set of partitions of~$\Omega$
which arise as the orbit partition of an element of~$G$.
The set of partitions of $\Omega$ is naturally ordered by refinement, and admits join and meet operations.
We say that $G$ is join-coherent if~$\pi(G)$ is join-closed, and
meet-coherent if~$\pi(G)$ is meet-closed.

Our central theorem states that the centralizer in~$\Sym(\Omega)$ of any
permutation~$g$ is meet-coherent, and subject to a certain finiteness condition on the
orbits of~$g$, also join-coherent. In particular, if $\Omega$ is a finite set then
the orbit partitions of elements of the centralizer in $\Sym(\Omega)$ of $g$ form a lattice.


A related result states that the intransitive direct product and the imprimitive wreath product of two
finite permutation groups are join-coherent if and only if each of the groups is join-coherent.
We also classify the groups $G$ such that $\pi(G)$ is a chain and prove two
further  theorems classifying the primitive join-coherent groups of finite
degree and the join-coherent groups of degree $n$ normalizing a subgroup generated
by an $n$-cycle.}
\end{abstract}

\maketitle

\thispagestyle{empty}

\section{Introduction}
Let~$G$ be a group acting on a set~$\Omega$.
Each element~$g$ of~$G$ has associated with it a partition~$\pi(g)$ of~$\Omega$, whose
parts are the orbits of~$g$. We define~$\pi(G)$ to be the set $\{\pi(g)\mid g\in G\}$.

If $\mathcal{P},\mathcal{Q}$
are two partitions of $\Omega$ then we say that $\mathcal{P}$ is a \emph{refinement} of $\mathcal{Q}$, and
write $\mathcal{P}\preccurlyeq \mathcal{Q}$, if every part of~$\mathcal{P}$ is contained in a part of~$\mathcal{Q}$.
The set of all partitions of $\Omega$ forms a lattice under~$\preccurlyeq$. That is to say,
any two partitions $\mathcal{P}$ and $\mathcal{Q}$ have a supremum with respect to refinement,
denoted $\mathcal{P} \join \mathcal{Q}$, and an infimum with respect to refinement, denoted $\mathcal{P} \meet
\mathcal{Q}$. The parts of $\mathcal{P} \meet \mathcal{Q}$ are precisely the non-empty intersections
of the parts of $\mathcal{P}$ and $\mathcal{Q}$; for a description of $\mathcal{P} \join \mathcal{Q}$
see Section~\ref{section:lattices} below. The lattice of all partitions of $\Omega$ is called the
\emph{congruence
lattice} on~$\Omega$, and is denoted $\Con(\Omega)$.

The object of this paper is to prove
a number of interesting structural and classification results
on
the permutation groups possessing one or both of the following 
 properties.


%

\begin{definition*}
Let~$G$ be a group acting on a set. 
\begin{defnlist}
\item We say that~$G$ is \emph{join-coherent} if~$\pi(G)$ is closed under~$\vee$.
\item We say that~$G$ is \emph{meet-coherent} if~$\pi(G)$ is closed under~$\wedge$.
\end{defnlist}
\end{definition*}



We refer to these properties collectively as \emph{orbit coherence} properties.
Our first main theorem describes the groups $G$ such that $\pi(G)$ is a chain.
It is clear that any such group is both join- and meet-coherent.

\setcounter{thmmain}{0}
\begin{thmmain}
\label{thmmain:chain}
Let $\Omega$ be a set, and let $G \le \Sym(\Omega)$ be such that $\pi(G)$ is a chain.
There is a prime~$p$ such that the length of any cycle of any element of $G$ is a power of $p$.
Furthermore,~$G$ is abelian, and either periodic or torsion-free.
\begin{thmlist}
\item If $G$ is periodic then $G$ is either a finite cyclic group of $p$-power order,
or else isomorphic to the \Prufer $p$-group.
\item If $G$ is torsion-free then $G$ is isomorphic to a subgroup of the $p$-adic rational numbers~$\Q_p$.
In this case $G$ has infinitely many orbits on $\Omega$, and the permutation group induced by its action
on any single orbit is periodic.
\end{thmlist}
\end{thmmain}

Our second theorem  determines when a direct product or wreath product
of permutation groups inherits join coherence from its factors.
The actions of these groups referred to in this theorem
are defined in Sections 4 and~5 below.


\begin{thmmain}
\label{thmmain:products}
Let $X$ and $Y$ be sets and let $G \le \Sym(X)$ and $H \le \Sym(Y)$ be  permutation groups.
\begin{thmlist}
\item If $G$ and $H$ are finite then
$G \times H$ is join-coherent in its product action on $X \times Y$ if and only if
$G$ and $H$ are 
join-coherent and have coprime orders.
\item If $Y$ is finite then
$G \wr H$ is join-coherent in its imprimitive action on $X \times Y$ if and only
if $G$ and $H$ are join-coherent.
\end{thmlist}
\end{thmmain} 



Our third main theorem, on centralizers in a symmetric group, is the central result of this paper.

\begin{thmmain}
\label{thmmain:Centralizers}
Let~$\Omega$ be a set, let $G = \Sym(\Omega)$, and let~$g\in G$. 
For $k\in\mathbf{N}\cup\{\infty\}$ let~$\n_k$ be the number of orbits of~$g$ of size~$k$. 
\begin{thmlist}
\item $\Cent_{G}(g)$ is meet-coherent. 
\item If~$\n_k$ is finite for all $k\neq 1$, including $k = \infty$,
then $\Cent_G(g)$ is join-coherent. 
\end{thmlist}
\end{thmmain}
We also show 
that
if the condition on the values $\n_k$ in the second part of the theorem
fails for a permutation $g \in \Sym(\Omega)$, then the centralizer in $\Sym(\Omega)$ of $g$ is
not join-coherent. Therefore this condition is necessary.

Theorem \ref{thmmain:Centralizers} implies, in particular,
that any centralizer in a finite symmetric group is both join- and meet-coherent.
This is a remarkable fact, and the starting point of our investigation, at least
chronologically. The observation that this important class of groups exhibits orbit coherence
justifies our study of these properties, and motivates the search for further examples.


The second part of the paper contains a partial classification of finite transitive join-coherent
permutation groups.
Our analysis depends on the fact that such a group
necessarily contains a full cycle, since the join of all the
orbit partitions of elements of a transitive permutation group is the trivial one-part partition.
The primitive permutation groups containing full cycles are known; 
we use this classification to prove the following theorem.

\begin{thmmain}\label{thmmain:primitive}
A primitive permutation group of finite degree 
is join-coherent if and only
if it is a symmetric group or a subgroup of $\AGL_1(\F_p)$ in its action on $p$ points,
where $p$ is prime.
\end{thmmain}

We also give a complete classification of the finite transitive join-coherent groups
in which the subgroup generated by a full cycle is normal.


\begin{thmmain}\label{thmmain:RNCS}
Let~$G$ be a permutation group on~$n$ points, containing a normal cyclic subgroup of order~$n$ acting regularly. Let~$n$ have prime factorization~$\prod_ip_i^{a_i}$.
Then~$G$ is join-coherent if and only if there exists for each~$i$ a transitive permutation
group~$G_i$ on~$p_i^{a_i}$ points, such that:
\begin{bulletlist}
\item[$\bullet$] if $a_i>1$ then~$G_i$ is either cyclic or the extension of a cyclic group of
order~$p_i^{a_i}$ by the automorphism
$x\mapsto x^r$ where $r = p_i^{a_i-1}+1$,
\item[$\bullet$] if $a_i=1$ then~$G_i$ is a subgroup of the Frobenius group of order $p(p-1)$,
\item[$\bullet$] the orders of the groups~$G_i$ are mutually coprime,
\item[$\bullet$] $G$ is permutation
isomorphic to the direct product of the groups~$G_i$ in its product action.
\end{bulletlist}
\end{thmmain}

Note that the permutation groups classified by Theorem \ref{thmmain:RNCS}
are always imprimitive, unless~$n$ is prime.
It would be interesting, but we believe difficult, to extend our results to a complete classification
of all finite transitive join-coherent permutation groups.
The principal obstruction to such a result is the apparently hard problem of
classifying
those transitive join-coherent
imprimitive permutation groups that
do not admit a non-trivial factorization as a direct product or a wreath product,
in the manner described in Theorem~\ref{thmmain:products}, and which do not normalize
a full cycle.
One example of such a group is the permutation group of degree $12$
generated by
\[ \renewcommand{\,}{\hskip3pt}
(1\,7)(4\,10), \quad (1\, 2\, 3\, 4\, 5\, 6\, 7\, 8\, 9\, 10\, 11\, 12).\]
It is not hard to check this group is an imprimitive join-coherent
subgroup of index $4$ in $C_4 \wr C_3$ and that it does not factorize
as a direct product or a wreath product.



In smaller degrees our results do yield a complete classification: every
join-coherent permutation group of degree at most $11$ is either a cyclic group acting
regularly, a symmetric group, one of the groups described in Theorem~\ref{thmmain:RNCS},
or an imprimitive wreath product of join-coherent groups of smaller degree,
as in Theorem~\ref{thmmain:products}(b).

The fact that there are no further join-coherent groups of degree at most~$11$, and also the join-coherence
of the group of degree $12$ presented above, have been verified computationally. In
Section~\ref{section:linear}, and
again in Section~\ref{section:primitive}, we require computer calculations to verify
that particular groups are not join-coherent. All of our computations have been performed
using Magma~\cite{Magma}. The code
for these computations is available from the second author's
website.\footnote{See \url{www.ma.rhul.ac.uk/~uvah099}.}

\subsection*{Further remarks and background}
\label{subsect:remarks}

It will be clear from the statement of our main results that the majority of them
concern join-coherence rather than meet-coherence. In part this is because
a finitely generated transitive
join-coherent permutation group must contain a full cycle, and the restriction
on the structure of the group that this imposes is very useful.
However an alternative characterization of join-coherence
suggests that it is a particularly natural property to study:
a permutation group~$G$ is join-coherent if and only if for every
finitely-generated subgroup~$H$ of~$G$, there exists an element $h\in G$ whose orbits are the same as the
orbits of~$H$. There is no similar characterization of meet-coherence in terms of subgroups.

There are groups which exhibit any combination of the properties
of join- and meet-coherence.
Any symmetric group is both join- and meet-coherent, but any non-cyclic alternating group is neither.
The group~\hbox{$C_2 \times C_2$} acting regularly on itself is meet- but not join-coherent.
Examples of groups that are join- but not meet-coherent are less easy to find, 
but one can check that the non-cyclic group of order~$21$,
in its action as a Frobenius group on~$7$ points, is such an example (see Section~\ref{section:Frobenius}
for our general results on Frobenius groups).

In the context of lattices, the operations~$\vee$ and~$\wedge$ are dual to one another.
This duality is not inherited to any great extent by the notions of join- and meet-coherence
of permutation groups. An asymmetry can be observed even
in
the congruence lattice of all partitions of a set:
compare for example the two parts of Lemma~\ref{lemma:updn} below.
For this reason, while there are some parts of the paper, for example Section~\ref{section:DP},
where join- and meet-coherence admit a common treatment, it is usually necessary to treat
each property separately.




Literature on the orbit partitions of permutations is surprisingly sparse. As we hope that this paper
shows, there are interesting general properties that remain to be discovered,
and we believe that further study is warranted.
One earlier investigation which perhaps has something of the same flavour is that of Cameron~\cite{CameronCycleClosed}
into cycle-closed permutation groups. If~$G$ is a permutation group on a finite set then the
\emph{cycle-closure}~$C(G)$ is the
group generated by all of the cycles of elements of~$G$. Cameron proves that any
group which is equal to its cycle-closure is isomorphic, as a permutation group,
to a direct product of symmetric groups in their natural action
and cyclic groups acting regularly, with the factors acting on disjoint sets.
He also shows that if~$G=G_0$, and $G_{i+1}=C(G_i)$, then $G_4=G_3$ and that there exist groups for which
$G_2\neq G_3$.

There is an extensive literature on the lattice of subspaces of a finite-dimensional vector space
invariant under a group of linear transformations.
In this context both the lattice elements and the lattice operations differ from ours, and
so there is no immediate connection to our situation.
Indeed, we show in Proposition~\ref{prop:Vjoin} below
that the general linear group $\GL(V)$ is never join-coherent when $\dim V > 1$,
except in the case when $V = \F_2^2$, in which case it acts on $V \setminus \{0\}$
as the full symmetric group.
Nonetheless, there are certain parallels that it is interesting to observe.
For instance, if~$T$ is an invertible
linear map on a $K$-vector space $V$, then by \cite[Theorem~2]{BrickmanFillmore},
the lattice of invariant subspaces of $V$
is a chain if and only if $T$ is cyclic of primary type.
Analogously, it follows from 
Theorem~\ref{thmmain:chain}
that if $G$ is a finite permutation
group then $\pi(G)$ is a chain if and only if $G$ is cyclic of prime-power order.
Thus if we regard~$T$ as a permutation of $V$, 
then~$\pi(\left<T\right>)$ is a chain
if and only if $T$ has prime-power order.
When $K$ has prime characteristic, this situation arises
either when $T- I$ is nilpotent,
or when there is a prime $p$ such that $V$ is a direct
sum of subspaces on which $T$
acts as a Singer element of $p$-power order. Such elements exist, for example, 
in $\GL_d(\mathbf{F}_{2^a})$ whenever $2^{ad} - 1$ is a prime power.
For an introduction
to the theory of invariant subspaces we refer the reader to~\cite{BrickmanFillmore}.


There are various areas of group theory in which lattices have previously arisen
which are not directly related to orbit partitions. Subgroup lattices, for example, have been well studied.
A well-known theorem of Ore~\cite{Ore} states that the
subgroup lattice of a group~$G$ is distributive if and only if~$G$ is locally cyclic.
Locally cyclic groups are also important in this paper: in
Proposition~\ref{prop:regularjoin} we show that they are precisely the groups
that are join-coherent in their regular action. A $p$-group is locally
cyclic if and only if it is a subgroup of the \Prufer $p$-group;
these groups appear
in Theorem~\ref{thmmain:chain}, as the class of transitive permutation groups $G$
such that~$\pi(G)$ is a chain.


\subsection*{Outline}
The outline of this paper is as follows.
In Section~\ref{section:lattices} we prove some general
results with a lattice-theoretic flavour that will be used throughout the paper.
We begin our structural results in Section~\ref{section:regular} where
we determine when a group acting regularly on itself is join- or meet-coherent. This section also contains
a proof of Theorem~\ref{thmmain:chain} on the
permutation groups $G$ for which~$\pi(G)$ is a chain.

Theorem~\ref{thmmain:products}
is proved for direct products in Proposition~\ref{prop:DPjoin}, 
and for wreath products
in Proposition~\ref{prop:imprimitive} 
and in Section~\ref{section:WP}.
Theorem~\ref{thmmain:Centralizers} on centralizers
is proved in Section~\ref{section:Centralizers}. We have chosen to offer logically independent
arguments in Section~\ref{section:WP} and Section~\ref{section:Centralizers}, even though
the results in these sections are quite closely connected.
This is partly so that they may be read independently, and partly because the two
lines of approach appear to offer different insights.

The second part
of the paper, which focusses on classification results, begins in
Section~\ref{section:Frobenius} where we classify join-closed Frobenius groups
of prime degree. In Section~\ref{section:linear} we determine when a linear group has a join-coherent
action on points or lines; these results are used in the proof of  Theorem~\ref{thmmain:primitive}
in Section~\ref{section:primitive}. Finally, Theorem~\ref{thmmain:RNCS} is proved
in Section~\ref{section:RNCS}.

To avoid having to specify common group actions every time they occur, we shall adopt the following conventions. Any group mentioned as acting on a set at its first appearance
will be assumed always to act on that set, unless another action is explicitly given. In particular,
$\Sym(\Omega)$ always
acts naturally on the set $\Omega$,
and the finite symmetric group~$S_n$ always acts on~$n$ points. 
The cyclic group~$C_k$, for $k \in \N$,
acts on itself by translation. All maps are written on the right.


\section{Partitions and imprimitive actions}
\label{section:lattices}

In this section we collect some facts about lattices of partitions that will be useful
in later parts of the paper.  For an introduction to the general theory, see for instance \cite{DaveyPriestley}.

Given a set partition $\mathcal{P}$ of a set~$\Omega$
we define a corresponding relation~$\equiv_\mathcal{P}$ on~$\Omega$ in which $x \equiv_\mathcal{P} y$ if and only if
$x$ and $y$ lie in the same part of $\mathcal{P}$. If $\mathcal{P}$ and $\mathcal{Q}$
are set partitions of $\Omega$ then it is not hard to see that $\mathcal{P} \join \mathcal{Q}$ is the set partition
$\mathcal{R}$ such that $\equiv_\mathcal{R}$ is the transitive closure of the relation
$\equiv$ defined on $\Omega$ by
\[ x \equiv y \iff \text{$x \equiv_\mathcal{P} y$ or $x \equiv_\mathcal{Q} y$.} \]
Similarly $\mathcal{P} \meet \mathcal{Q}$ is the set partition $\mathcal{R}$
such that
\[ x \equiv_\mathcal{R} y \iff \text{$x \equiv_\mathcal{P} y$ and $x \equiv_\mathcal{Q} y$}. \]
Equivalently, as we have already remarked,
\[ \mathcal{P} \meet \mathcal{Q} = \{ P \cap Q \mid P \in \mathcal{P},
Q \in \mathcal{Q}, P \cap Q \not= \varnothing \}. \]

The congruence lattice $\Con(\Omega)$ of all partitions of $\Omega$ is
distributive, i.e.~the lattice operations~$\vee$ and~$\wedge$ distribute over one another.
In the language of lattice theory, a permutation group $G \le \Sym(\Omega)$ is join-coherent if and only if
 $\pi(G)$ is an upper subsemilattice of $\Con(\Omega)$, meet-coherent if and only if
$\pi(G)$ is a lower subsemilattice of $\Con(\Omega)$, and both join- and meet-coherent
if and only if $\pi(G)$ is a sublattice of $\Con(\Omega)$.


\begin{lemma}\label{lemma:distributivity}
Let~$L$ be a distributive lattice with respect to $\preccurlyeq$, and let $x\in L$.
\begin{numlist}
\item Define
\[
\mathrm{Up}(x)=\{y\in L \mid x\preccurlyeq y\},
\]
\[
\mathrm{Dn}(x)=\{y\in L \mid y\preccurlyeq x\}.
\]
Then $\mathrm{Up}(x)$ and $\mathrm{Dn}(x)$ are sublattices of~$L$.
\item The maps $\phi^x:L\longrightarrow \mathrm{Up}(x)$ and $\phi_x:L\longrightarrow \mathrm{Dn}(x)$
defined by
\[
y\phi^x = y\vee x,\quad y\phi_x = y\wedge x,
\]
are lattice homomorphisms.
\end{numlist}
\end{lemma}

\begin{proof}
The first part follows directly from the defining property of
$\join$ and $\meet$, and the second part from the definition of distributivity.
\end{proof}

Note that in part (1) of the following lemma, $\Con(\mathcal{B})$ is the congruence lattice
on the set \hbox{$\{B \mid B \in \mathcal{B} \}$} of parts of a partition $\mathcal{B}$.

\begin{lemma}\label{lemma:updn}
Let $\Omega$ be a set and let $\mathcal{B}\in\Con(\Omega)$.

\begin{numlist}
\setcounter{numlistcnt}{0}
\item $\mathrm{Up}(\mathcal{B}) \cong \Con(\mathcal{B})$.
\item $\mathrm{Dn}(\mathcal{B}) \cong \prod_{B\in\mathcal{B}}\Con(B)$.
\end{numlist}
\end{lemma}
\begin{proof}
If $\mathcal{B}\preccurlyeq\mathcal{A}$ then each part of~$\mathcal{A}$
  is a union of parts of~$\mathcal{B}$. Hence
$\mathcal{A}$ determines and is determined by a partition of~$\mathcal{B}$;
this gives a
bijection between $\mathrm{Up}(\mathcal{B})$ and $\Con(\mathcal{B})$
that is a lattice isomorphism.
For the second part we note that whenever $\mathcal{A}\preccurlyeq\mathcal{B}$,
each part~$B$ of~$\mathcal{B}$ is a union of parts of~$\mathcal{A}$, and so~a
subset of the parts of $\mathcal{A}$ form a partition of~$B$.
Clearly~$\mathcal{A}$ itself is determined by these partitions of the parts of~$\mathcal{B}$,
and thus~$\mathcal{A}$ determines and is determined by an element of $\prod_{B\in\mathcal{B}}\Con(B)$.
It is easy to see that this bijection is a lattice isomorphism. 
\end{proof}

The next proposition is a straightforward consequence of Lemma~\ref{lemma:updn}.
As a standing
convention, we avoid the use of the word `respectively' when the same short statement or proof
works for either a join- or a meet-coherent group.

\begin{proposition}\label{prop:partitions}
 Let~$G$ be a join- or meet-coherent permutation group on~$\Omega$.
\begin{numlist}
\item Let~$\mathcal{B}$ be a partition of~$\Omega$, and let~$H$ be the group of permutations which fix
every part of~$\mathcal{B}$ set-wise. Then $G\cap H$ is join- or meet-coherent. 
\item Let $X\subseteq \Omega$, and let~$H$ be the set-stabilizer of~$X$ in~$G$. Then~$H$ is join- or meet-coherent.
\item Any point-stabilizer in~$G$ is join- or meet-coherent.
\end{numlist}
\end{proposition}
\begin{proof}
If $h \in \Sym(\Omega)$ then
 $\pi(h)\preccurlyeq \mathcal{B}$ if and only if $Bh=B$ for all $B\in\mathcal{B}$.
Thus $\pi(H)=\mathrm{Dn}(\mathcal{B})$, which is closed under~$\vee$ and~$\wedge$.
The first part of proposition follows, since the intersection of two join- or meet-closed sets
is join- or meet-closed. The second part follows from the first by taking
$\mathcal{B}=\{X,\Omega\setminus X\}$, and the third part follows 
from the second by taking~$X$ to be a singleton set.
\end{proof}

If $G$ is a permutation group acting on a set $\Omega$, then
we may consider the natural action of $G$ on $\Con(\Omega)$,
defined for $g \in G$ and $\mathcal{P} \in \Con(\Omega)$
by \hbox{$\mathcal{P}^g = \{ P^g \mid P \in \mathcal{P} \}$}. This action will be used
in Lemma~\ref{lemma:CentCrit} below, which is the critical step in the proof
of Theorem~\ref{thmmain:Centralizers}.

Recall that a transitive permutation group~$G$ on~$\Omega$ is said to be
\emph{imprimitive} if
it stabilizes a non-trivial partition $\mathcal{B}$ of $\Omega$,
in the sense that $xg\partequiv_\mathcal{B} yg
\Longleftrightarrow
x\equiv_\mathcal{B} y$ for all $x,y\in\Omega$ and $g\in G$.
An equivalent restatement, using the action just defined, is that $\mathcal{B}^g = \mathcal{B}$
for each $g \in G$.
In this case, one says that $\mathcal{B}$ is a \emph{system of imprimitivity} for the action of $G$ on $\Omega$.
Otherwise $G$ is \emph{primitive}.
If $\mathcal{B}$ is a $G$-invariant partition of $\Omega$ then
$\mathcal{B}$ inherits an action of~$G$,
since for $B \in \mathcal{B}$ and $g \in G$ we have $B^g \in \mathcal{B}$.

The next proposition gives the easier direction in Theorem~\ref{thmmain:products}(b). The proof
of this theorem is completed in Proposition~\ref{prop:WPjoinc}.

\begin{proposition}\label{prop:imprimitive}
Let $G$ be  a join- or meet-coherent permutation group on $\Omega$,
and let $\mathcal{B}$ be a system of imprimitivity for the action.
\begin{numlist}
\item The action of $G$ on $\mathcal{B}$ is join- or meet-coherent.
\item The set-stabilizer in $G$ of a part $B$ of $\mathcal{B}$ acts join- or meet-coherently on~$B$.
\end{numlist}
\end{proposition}
\begin{proof}
The second part is immediate from Proposition \ref{prop:partitions}(2). For the first part,
let \hbox{$\theta:\mathrm{Up}(\mathcal{B})\longrightarrow \Con(\mathcal{B})$} be the isomorphism
in Lemma \ref{lemma:updn}(1). By Lemma~\ref{lemma:distributivity}(2),
the composite map $\phi^\mathcal{B}\theta:\Con(\Omega)\longrightarrow \Con(\mathcal{B})$ is a homomorphism.
It is easy to check that if $g\in G$ has orbit partition~$\mathcal{P}$ on~$\Omega$ then $g$ has
orbit partition~\hbox{$\mathcal{P}\phi^\mathcal{B}\theta = (\mathcal{P} \join \mathcal{B})\theta$}
on~$\mathcal{B}$. The result now follows from 
the fact that 
$\Con(\mathcal{B})$ is distributive.
\end{proof}


\section{Regular representations and chains}
\label{section:regular}

Let~$G$ be a group of permutations of a set~$\Omega$. We say that~$G$ acts \emph{semiregularly} if every element of~$\Omega$ has a trivial point stabilizer in~$G$.
This is equivalent to the condition that for every element~$g\in G$,
all of the parts of the orbit partition~$\pi(g)$ of~$\Omega$ are of the same size.
We say that the action of~$G$ is \emph{regular} if it is semiregular and transitive.

\begin{proposition}\label{prop:RegMeet}
A group $G$ acting semiregularly is meet-coherent.
\end{proposition}
\begin{proof}
Suppose that~$G$ acts semiregularly on~$\Omega$. Let $x$, $y\in G$,
and let~$z$ be a generator for the cyclic
group $\langle x \rangle \cap \langle y \rangle$. We shall show that $\pi(z) = \pi(x) \wedge \pi(y)$.

Let $u\in \Omega$, and let~$P$ and~$Q$ be the parts of~$\pi(x)$ and~$\pi(y)$ respectively which contain~$u$.
Then
\[ P = \{ux^i \mid i  \in \Z \} \quad \text{and} \quad
   Q = \{uy^j \mid j  \in \Z \}.
\]
Let $v \in P \cap Q$, and let $i$, $j\in \Z$ be such that $v = ux^i = uy^j$.
By the semiregularity of~$G$, we have $x^i=y^j$, and so $x^i$, $y^j\in \langle z\rangle$.
It follows that $P\cap Q = \{ u z^k \mid k \in \Z \}$, and
hence that $P \cap Q$ is a part of~$\pi(z)$.
\end{proof}

We recall that a group~$G$ is said to be \emph{locally cyclic} if any pair of elements of~$G$
generate a cyclic subgroup.
The following lemma states a well-known fact.
\begin{lemma}\label{lemma:locallycyclic} A group is locally cyclic if and only if it is isomorphic to a
section of~$\Q$.
\end{lemma}
\begin{proof}
It is clear that a locally cyclic group is either periodic or torsion-free.
If it is periodic then it has at most one subgroup of order~$n$ for each $n\in\N$, and it is easy
to see that it embeds into the quotient group $\Q / \Z$; for torsion-free groups the
result was first proved in \cite{Baer}.\footnote{We thank Mark Sapir for this reference.}
\end{proof}

\begin{proposition}\label{prop:regularjoin}
A group $G$ acting regularly is join-coherent if and only if it is locally cyclic.
\end{proposition}
\begin{proof}
We may assume without loss of generality that~$G$ is transitive, and so we may suppose that it acts on itself by translation.
Let $x$, $y\in G$, and let $H=\langle x,y\rangle$.
The partition $\pi(x) \join \pi(y)$ is precisely the partition
of~$G$ into cosets of~$H$. If $H=\langle z\rangle$ then we have $\pi(x)\join \pi(y)=\pi(z)\in\pi(G)$.
On the other hand, if~$H$ is not cyclic then it cannot be a part of~$\pi(z)$ for any $z\in G$, and so
$\pi(x)\vee\pi(y)\notin\pi(G)$.
\end{proof}

We define a locally cyclic group of particular importance to us.
\begin{definition*}
The \Prufer $p$-group, $\P$, is the subgroup of $\Q/\Z$ generated by
the cosets containing $1/p^i$ for $i\in\N$.
\end{definition*}
The group $\P$ appears in the statement of Theorem~\ref{thmmain:chain}, whose proof occupies the
remainder of this section. We begin with the following proposition.


\newcommand{\GO}{G_\mathcal{O}}
\begin{proposition}\label{prop:chain1}
Let $\Omega$ be a set, and let $G\le\Sym(\Omega)$ be such that $\pi(G)$ is a chain.
\begin{numlist}
\item No element of $G$ has an infinite cycle.
\item Let $\mathcal{O}$ be an orbit of $G$, and let $G_{\mathcal{O}}\le\Sym(\mathcal{O})$
be the permutation group induced by
the action of $G$ on $\mathcal{O}$. Then $\GO$ acts regularly on~$\mathcal{O}$.
\item There is a prime $p$ such that every cycle of every element of $g$ has $p$-power length.
\item If $G$ acts transitively, then there is a prime $p$ such that $G$ is
isomorphic to a subgroup of the Pr{\"u}fer $p$-group $\P$.
\end{numlist}
\end{proposition}
\begin{proof}
It is clear that if $g \in G$ has an infinite cycle then
$\pi(g^a)$ and $\pi(g^b)$ are incomparable whenever $a$ and $b$
are natural numbers such that neither divides the other. Clearly this implies~(1).

For (2), let $g$ be an element of $G$ which acts non-trivially on $\mathcal{O}$, and suppose that $g$ has
a fixed point $x\in\mathcal{O}$. Let $z\in\mathcal{O}$ be such that $zg\neq z$, and
let $h\in G$ be such that $xh = z$. Then $g^h$ has $z$ as a fixed
point, and $xg^h \not= x$. Hence the partitions $\pi(g)$ and $\pi(g^h)$ are incomparable, a contradiction.
Therefore $g$ has no fixed points on $\mathcal{O}$, and it follows that the action of $\GO$ is regular.

For (3), suppose that there exist distinct primes $q$ and $r$
such that~$\pi(g)$ has a part $Q$ of size divisible by $q$, and another
part $R$ of size divisible by~$r$. Let $m$ be the order of the permutation
induced by $g$ on $Q \cup R$.
Then~$\pi(g^{m/r})$
has singleton parts corresponding to the elements of $Q$,
and $R$ is a union of parts of~$\pi(g^{m/r})$ of size at least $2$.
A similar remark holds for~$\pi(g^{m/q})$ with $Q$ and $R$ swapped, and
so~$\pi(g^{m/q})$ and $\pi(g^{m/r})$
are not comparable; again this is a contradiction.
It follows that for each
$g\in G$, there is a prime $p_g$ such
that every cycle of $g$ has length a power~of~$p_g$.

We now show that $p_g=p_h$ for all non-identity permutations $g$, $h\in G$.
We may suppose without loss of generality that $\pi(h)\preccurlyeq\pi(g)$, and so each orbit
of $\langle g\rangle$
is a union of orbits of $\langle h\rangle$. But by (2), $h$ acts regularly
on each of its orbits. There exists a $\langle g\rangle$-orbit $\mathcal{O}$
on which $h$ acts non-trivially, and so
$p_h$ divides $|\mathcal{O}|$. But $|\mathcal{O}|$ is a power of $p_g$, and so $p_h=p_g$.

For (4), we note that if $G$ is transitive, then by (2) it acts regularly. By Proposition
\ref{prop:regularjoin}, $G$ is locally cyclic. By Lemma \ref{lemma:locallycyclic} it follows that
$G$ is isomorphic to a section of $\Q$. Now an element $g\in G$ has only finite cycles by (1), and since $G$ acts
regularly, all of the cycles of $g$ have the same length. Therefore $g$ has finite order, and so $G$ is not torsion-free. Hence~$G$ is isomorphic to a subgroup of $\Q/\Z$.
Write $g$ as $a/b+\Z$, where $a$ and $b$ are
coprime. Since every cycle of $g$ has~$p$-power length by (3),
the denominator~$b$ must be a power of~$p$; hence $G$ is isomorphic to
a subgroup of the Pr\"ufer $p$-group.
\end{proof}

We note that any subgroup of $\P$ is either cyclic of $p$-power order, or
equal to $\P$ itself. Therefore (4) implies that
any finite group whose orbit partitions form a chain is cyclic of prime-power order.

There are interesting
examples of groups $G$ acting intransitively on an infinite set,
such that~$G$ is not locally cyclic,
but $\pi(G)$ is a chain. Let $\alpha$ be an irrational element of the $p$-adic integers $\Z_p$, such
that $p$ does not divide~$\alpha$, and define $\alpha_i \in \{0,1,\ldots,p^i-1\}$ by
$\alpha_i = \alpha \bmod p^i$.
For instance, we may take $p = 2$ and $\alpha = 1010010001000010\ldots \in \Z_2$;
the sequence $(\alpha_i)$ here is
\[
(1, 1, 5, 5, 5, 37, 37, 37, 37, 549, \ldots ).
\]
Let $\Omega$ be an infinite set
and let $\{c_i\mid i\in\N\}$ be a set of mutually disjoint cycles in $\Sym(\Omega)$,
such that $c_i$ has length $p^i$.
Let $G \le \Sym(\Omega)$ be generated by $g$ and $h$, where $g = \prod_{i=1}^\infty c_i$
and $h = \prod_{i=1}^\infty c_i^{\alpha_i}$. Then $g$ and $h$ have the same orbit partition,
but it is easily seen that they are not powers of one another. Hence $G$ is not cyclic.
However, the permutation group induced by~$G$ on any finite set of its orbits is cyclic,
since for any given~$j$ there exists $\beta_j \in \N$ such that $\alpha_j \beta_j \equiv 1 \bmod p^j$,
and so $h^{\beta_j}$ agrees with $g$ on the orbits of all of the cycles $c_i$ for $i \le j$. It
follows easily that $\pi(G)$ is a chain.

The group in this example falls under (2) in Theorem~\ref{thmmain:chain}; to prove
this theorem we shall use Proposition~\ref{prop:chain1} and the following
lemma.\footnote{The authors would like to thank Benjamin Klopsch and
John MacQuarrie for helpful conversations on this subject.}

\begin{lemma}\label{lemma:Pruferlimit}
Let $p$ be a prime and let $I$ be a totally ordered set. For each $i \in I$, let $P_i$ be an isomorphic
copy of the \Prufer $p$-group $\P$. Let
\[
\{f_{ji}:P_j\rightarrow P_i \mid i,j\in I, i \le j\}
\]
be a set of non-zero homomorphisms, with the property that
$f_{kj}f_{ji} = f_{ki}$  whenever $i\le j\le k$.
Let $M$ be the inverse limit
$\varprojlim P_i$, taken with respect to the totally ordered set $I$ and the homomorphisms $f_{ji}$.
If all but finitely many of the $f_{ji}$ are isomorphisms then $M \cong \P$,
and otherwise $M \cong \mathbf{Q}_p$, the additive group of $p$-adic rational numbers.
\end{lemma}
\begin{proof}
If $f : \P \rightarrow \P$ is an endomorphism,
then for each $i \in \N$, there exists a unique $a_i \in \{0,1,\ldots, p^i-1\}$ such
that
\begin{equation}\tag{$\star$}
(x/p^i + \Z) f = a_i (x / p^i + \Z) \quad\textrm{\ for all\ } x \in \Z.
\end{equation}
It is easily seen
that if $i \le j$ then $a_j \equiv a_i$ mod $p^i$. Therefore if $a \in \Z_p$ is the $p$-adic
integer such that $a \equiv a_i$ mod $p^i$ for each $i \in \N$, then $f$
is the map $\mu(a) : \P \rightarrow \P$ defined by ($\star$) above.
It follows that at most a countable infinity of the maps $f_{ji}$ are non-isomorphisms.

Observe that $\mu(a)$ is surjective unless $a=0$, and an isomorphism if and only
if $a$ is not divisible by $p$. If all but finitely many of the $f_{ji}$ are isomorphisms,
then it is clear that $M$ is isomorphic to $\P$. Otherwise there exists an infinite
increasing sequence $(i_k)$ of elements of $I$,
such that if we set $R_k = P_{i_k}$ and
\[ g_k = f_{i_{k+1} i_k} : R_{k+1} \rightarrow R_k \]
for $k \in \N$, then each $g_k$ is a non-isomorphism, and $M \cong \varprojlim R_k$.

For each $k \in \N$, let $a_k \in \Z_p$ be such that $g_k = \mu(a_k)$. Let
$a_k = p^{e_k} b_k$ where~$p$ does not divide~$b_k$; by assumption $e_k \ge 1$ for
each $k$. Let $d_k = \prod_{i=1}^{k-1} b_i$ for each $k \in \N$. Then
in the commutative diagram

\vspace{-10pt}
\begin{equation}\tag{$\dagger$}
\xymatrix{
R_1 \ar@{=}[d] & R_2 \ar[l]_{\mu( a_1)} \ar[d]^{\mu( d_2)} & R_3 \ar[l]_{\mu( a_2)} \ar[d]^{\mu( d_3)}
    & R_4 \ar[l]_{\mu( a_3)} \ar[d]^{\mu( d_4)} & \ar[l]_{\mu( a_4)} \cdots \\
R_1 & R_2 \ar[l]^{\mu( p^{e_1})} & R_3 \ar[l]^{\mu( p^{e_2})}  & R_4 \ar[l]^{\mu( p^{e_3})}
& \ar[l]^{\mu( p^{e_4})} \cdots
 }
%
\end{equation}
\vskip6pt

\noindent all of the vertical arrows are isomorphisms. It follows easily that
the inverse limits constructed with respect to the
top and bottom rows are isomorphic. Moreover, the inverse system
\[ \xymatrix{
P_1 & P_2 \ar[l]_{\mu(p)} & P_3 \ar[l]_{\mu(p)}  & P_4 \ar[l]_{\mu(p)}
& \ar[l]_{\mu(p)} \cdots} \]
in which all of the maps are multiplication by $p$ is a refinement of
the bottom row of the diagram ($\dagger$), and so it defines the same inverse limit.

Finally we note that $\P \cong \Q_p / \Z_p$, and that after applying
this isomorphism, the map $\mu(p) : \P \rightarrow \P$ is induced by
multiplication by $p$ in $\Q_p$. Hence $M \cong \varprojlim \Q_p / p^k \Z_p$.
Since $p^k \Z_p$ is an open subgroup of 
$\Q_p$, and since $\bigcap_k p^k \Z_p = \{0\}$, it follows that $M \cong \Q_p$ as required.
\end{proof}

We are now ready to prove Theorem~\ref{thmmain:chain}, which we restate below for convenience.


\setcounter{thmmain}{0}
\begin{thmmain}
Let $\Omega$ be a set, and let $G \le \Sym(\Omega)$ be such that $\pi(G)$ is a chain.
There is a prime~$p$ such that the length of any cycle of any element of $G$ is a power of $p$.
Furthermore,~$G$ is abelian, and either periodic or torsion-free.
\begin{thmlist}
\item If $G$ is periodic then $G$ is either a finite cyclic group of $p$-power order,
or else isomorphic to the \Prufer $p$-group.
\item If $G$ is torsion-free then $G$ is isomorphic to a subgroup of the $p$-adic rational numbers~$\Q_p$.
In this case $G$ has infinitely many orbits on $\Omega$, and the permutation group induced by its action
on any single orbit is periodic.
\end{thmlist}
\end{thmmain}

\begin{proof}
By Proposition \ref{prop:chain1}(4), there is a prime $p$ such that $G$ acts as a subgroup
of the Pr{\"u}fer $p$-group $\P$ on each of its orbits. It follows that $G$ is abelian.
Suppose that $g \in G$ has infinite order, and $h \in G$ has finite order~$p^a$. Then 
$\pi(h)\prec \pi(g^i)$ for any $i\in\N$, since $g$ has cycles of length greater than $ip^a$.
But each cycle of $g$ is finite, and so every element of $\Omega$ is fixed by some power
of $g$. Hence $h$ is the identity of $G$. This shows that $G$ is either torsion-free
or periodic.

Suppose that $K_1$ and $K_2$ are the kernels of the action of $G$ on distinct orbits.
If there exists $k_1 \in K_1 \setminus K_2$ and $k_2 \in K_2 \setminus K_1$ then
the orbit partitions $\pi(k_1)$ and $\pi(k_2)$ are clearly incomparable.
Hence the kernels of the action of $G$ on its various orbits form a chain
of subgroups of $G$. Since $G$ acts faithfully on $\Omega$,
the intersection of all these kernels is trivial.

Suppose that $G$ is periodic. Let $g_1,\dots,g_r\in G$ be a finite collection of elements, and let
$H=\langle g_1,\dots,g_r\rangle$. Then $H$ is finite.
Since $H$ satisfies the descending chain condition on subgroups,
there is an orbit $\mathcal{O}$ of $G$ for which the kernel~$K_\mathcal{O}$ of $G$
acting on $\mathcal{O}$
intersects trivially with~$H$. Therefore $H$ is isomorphic to a subgroup of
$G / K_\mathcal{O}$. It now follows from Proposition~\ref{prop:chain1}(4) that~$H$
is isomorphic to a subgroup of $\P$. Hence any finitely
generated subgroup of $G$ is cyclic, and so $G$ is locally cyclic.
Now by Lemma~\ref{lemma:locallycyclic} we see that $G$ itself is isomorphic to a
subgroup of $\P$.

The remaining case is when $G$ is torsion-free.
Let $\mathcal{O}_i$ for $i \in I$
be the set of orbits of $G$, where~$I$ is a suitable indexing set, and let $K_i$
be the kernel of $G$ acting on $\mathcal{O}_i$. Order $I$ so that for $i$, $j \in I$,
we have $i \le j$ if and only if $K_i \ge K_j$. For $i\le j$ let $f_{ji} : G / K_j \rightarrow
G / K_i$ be the canonical surjection.
Fix for each $i\in I$ an isomorphism $G / K_i \rightarrow P_i$ where $P_i\cong \P$.
Since $\bigcap_{i \in I} K_i = \{1_G\}$
and each $G / K_i$ is isomorphic to a subgroup of $\P$,
we see that $G$ is isomorphic to a subgroup of the inverse limit
$\varprojlim P_i$, taken with respect to the totally ordered set $I$ and the homomorphisms
$f_{ji}$. The theorem now follows from Lemma~\ref{lemma:Pruferlimit}.
\end{proof}

We end this section by remarking
that if $G\le\Sym(\Omega)$ is isomorphic to a subgroup of $\P$, then it has an orbit on $\Omega$ on which it acts faithfully and regularly.
For if $G$ is non-trivial then the intersection of the non-trivial subgroups of $G$ has order $p$; now since $G$ acts faithfully on~$\Omega$,
it follows that there is an orbit $\mathcal{O}$ on which $G$ acts with trivial kernel. Since $G$ is abelian, its action on $\mathcal{O}$ is regular.

\section{Direct products}\label{section:DP}

Suppose that $G$ and $H$ are groups acting on sets $X$ and $Y$ respectively. There are two natural actions of the
direct product $G\times H$, namely the \emph{intransitive action} on the disjoint union $X\dotcupop Y$
 and the \emph{product action} on $X\times Y$.
For $(g,h) \in G \times H$ the intransitive action is defined by
\[
z(g,h) = \begin{cases} zg & \text{if $z \in X$} \\ zh & \text{if $z \in Y$,} \end{cases}
\]
where $z \in X \dotcupop Y$, and the product action is defined by
$(x,y)(g,h) = (xg,yh)$ where $(x,y) \in X \times Y$.
The product action is the subject of Theorem~\ref{thmmain:products}(a), which we prove
in this section. Both of these actions also arise in later parts of the paper. 

\begin{lemma}\label{lemma:DPorbits}
Let $X$ and $Y$ be sets.
Let $g\in\mathrm{Sym}(X)$ and $h\in\mathrm{Sym}(Y)$, and let $k=(g,h)\in\mathrm{Sym}(X)\times\mathrm{Sym}(Y)$.
\begin{numlist}
\item \label{DPorbitsDU} In its action on $X\dotcupop Y$ we have $\pi(k)=\pi(x)\dotcupop \pi(y)$.
\item \label{DPorbitsDP} 
Suppose that $g$ and $h$ have finite coprime orders. Then in its action on $X\times Y$ we have
$\pi(k)=\{P\times Q\mid P\in \pi(g),\, Q\in \pi(h)\}$.
\end{numlist}
\end{lemma}
\begin{proof}
Both parts are straightforward.
\end{proof} 

The intransitive  action of a direct product is dealt with in the next proposition.

\vbox{
\begin{proposition}\label{prop:DPc}
Let $G$ and $H$ be groups acting on sets $X$ and $Y$ respectively.
\begin{numlist}
\item The action of $G\times H$ on $X \dotcupop Y$ is join-coherent if and only if the actions of $G$ and $H$ are
join-coherent.
\item The action of $G\times H$ on $X \dotcupop Y$ is meet-coherent if and only if the actions of $G$ and $H$ are
meet-coherent.
\end{numlist}
\end{proposition}

\begin{proof}
Suppose that $g_1,g_2,g_3 \in G$ and $h_1,h_2,h_3\in H$. It is clear from
Lemma~\ref{lemma:DPorbits}(\ref{DPorbitsDU}) that
\[ \pi(g_1,h_1)\vee\pi(g_2,h_2)=\pi(g_3,h_3) \]
if and only if $ \pi(g_1)\vee\pi(g_2)=\pi(g_3)$ and $ \pi(h_1)\vee\pi(h_2)=\pi(h_3)$,
and similarly
\[ \pi(g_1,h_1)\wedge\pi(g_2,h_2)=\pi(g_3,h_3) \]
if and only if $\pi(g_1)\wedge\pi(g_2)=\pi(g_3)$ and $\pi(h_1)\wedge\pi(h_2)=\pi(h_3)$.
The result follows.
\end{proof}}


Turning to the product action, we prove Theorem~\ref{thmmain:products}(a).
For convenience we restate this result as the following proposition.

\begin{proposition}\label{prop:DPjoin}
Let $X$ and $Y$ be sets and let $G \le \Sym(X)$ and $H \le \Sym(Y)$ be
finite permutation groups.
Then $G\times H$ is join-coherent on $X\times Y$ if and only if
$G$ and $H$ are both join-coherent, and have coprime orders.
\end{proposition}
\begin{proof}
It is obvious that if either $G$ or $H$ is not join-coherent, then neither is $G\times H$.
So we suppose that $G$ and $H$ act join-coherently. If $|G|$ and $|H|$ are coprime then the join-coherence of
$G\times H$ on $X\times Y$ follows from
Lemma~\ref{lemma:DPorbits}(\ref{DPorbitsDP}).

Suppose conversely that there is a prime $p$ which divides both $|G|$ and~$|H|$.
Let $g\in G$ have order $p^a$, and let $h\in H$ have
order $p^b$, where these are the largest orders of $p$-elements in each group. Then it is clear that
$\pi(g)$ has a part of size $p^a$, and $\pi(h)$ a part of size $p^b$. Now $\pi(g,1)\vee\pi(1,h)$
has a part of size
$p^{a+b}$, and it follows that it cannot be in $\pi(G\times H)$, since the greatest order of a $p$-element in $G\times H$ is $\mathrm{max}(p^a,p^b)$.
Hence if $G \times H$ is join-coherent then
$G$ and $H$ have coprime orders.
\end{proof} 

The following result may be viewed as a partial converse to Proposition~\ref{prop:DPjoin}.
It is used in the proof of Theorem~\ref{thmmain:RNCS} in Section~\ref{section:RNCS} below.


\begin{proposition}\label{prop:DPconverse}
Suppose that $G$ acts on finite sets $X$ and $Y$, and that these actions have kernels $K_X$ and $K_Y$ respectively, where $K_X\cap K_Y=1$.
If~$G$ is join-coherent on $X\times Y$ then
\begin{numlist}
\item\label{DPc1} $K_Y$ is join-coherent on $X$ and $K_X$ is join-coherent on $Y$;
\item\label{DPc2} $|K_X|$ and $|K_Y|$ are coprime;
\item\label{DPc3}  $G=K_YK_X\cong K_Y\times K_X$.
\end{numlist}
\end{proposition} 

\begin{proof}
Consider the partition $\mathcal{B}$ of $X\times Y$ into parts $X\times\{y\}$ for $y\in Y$.
Clearly~$K_Y$ is the largest subgroup of $G$ which stabilizes the parts of $\mathcal{B}$, and its action on
each part is that of~$K_Y$ on~$X$. Since $K_X \cap K_Y = 1$,
this action is faithful.
Hence, by Proposition~\ref{prop:partitions}(1),
the action of~$K_Y$ on $X$ is join-coherent, and (1) follows.

Let $p$ be a prime, and let $p^a$ and $p^b$ be the largest powers of $p$
dividing $|G/K_X|$ and $|G/K_Y|$ respectively.
If $P$ is a Sylow $p$-subgroup of $G$, then $P$ contains elements $g_X$ and $g_Y$ such that~$g_XK_X$ has
order $p^a$ in $G/K_X$ and $g_YK_Y$ has order $p^b$ in $G/K_Y$. It follows
that~$g_X$ has an orbit $\mathcal{O}_X$ on $X$ of size $p^a$, and that
$g_Y$ has an orbit $\mathcal{O}_Y$ on $Y$ of size $p^b$.
Now $\langle g_X,g_Y\rangle$ is a subgroup of $P$, and so its orbits on $X\times Y$ have $p$-power order.
One of these orbits contains $\mathcal{O}_X\times\mathcal{O}_Y$, and so has order $p^c$ for some $c\ge a+b$.

Since $G$ is join-coherent on $X\times Y$, it must contain
an element $g$ whose orbits are those of $\langle g_X,g_Y\rangle$. It is clear that $g$ must
be a $p$-element of
order at least $p^{a+b}$. Since $G/K_X$ and $G/K_Y$ have $p$-exponents $p^a$ and $p^b$ respectively,
it follows that $g^a \in K_X$, $g^b \in K_Y$ and so
\[ g^{\max(a,b)}\in K_X\cap K_Y = 1. \] 
Hence
one of $a$ or $b$ is $0$ and $G/K_X$ and $G/K_Y$ have coprime orders.

Suppose that $p^r$ divides $|G|$. Since at least one of $G/K_X$ or $G/K_Y$ has order coprime with~$p$, we see that $p^r$ must divide one of $|K_X|$ or $|K_Y|$, and it follows that
$|G|$ divides $|K_X|\cdot |K_Y|$. But since $K_X\cap K_Y=1$ 
we have $|K_YK_X|=|K_Y|\cdot|K_X|$, and hence $G=K_YK_X$; since $K_X$ and $K_Y$ are both normal, this implies that
$G\cong K_Y\times K_X$, as stated in (\ref{DPc3}). Now (\ref{DPc2}) follows from the
final sentence of the previous paragraph.
\end{proof}

Necessary and sufficient criteria for the meet-coherence of $G\times H$
in its product action
are harder to obtain. Since we shall not need any such results in later parts of the paper, we merely
offer the following partial result.
\begin{proposition}
Let $X$ and $Y$ be sets and let $G \le \Sym(X)$ and $H \le \Sym(Y)$ be meet-coherent
permutation groups.  If $G$ and $H$ are finite of coprime order,
then $G\times H$ is meet-coherent on $X\times Y$.
\end{proposition}
\begin{proof}
This follows easily from Lemma \ref{lemma:DPorbits}(2).
\end{proof}
There are examples of groups $G \le \Sym(X)$ and $H \le \Sym(Y)$,
not of coprime order, such that $G\times H$ is meet-coherent on $X\times Y$. For instance if
both $G$ and $H$ act semiregularly, then so does~$G\times H$, and so by Proposition~\ref{prop:RegMeet} we see that  $G\times H$ is meet-coherent regardless
of the orders of $G$ and $H$.

\section{Wreath products}\label{section:WP}

Given sets $S$ and $T$, we write $S^T$ for the set of maps from $T$ to $S$. As usual,
we shall write all maps on the right.
If $S$ is a group, then $S^T$ inherits a group structure as the direct
product of~$|T|$ copies of $S$; here $|T|$ may be infinite. 


Throughout this section we let $G$ and $H$ be groups acting on sets $X$ and~$Y$ respectively.
The \emph{unrestricted wreath product} $G\wr H$ is defined to be the semidirect product $G^Y\rtimes H$, where
the action of $H$ on $G^Y$ is given by $fh=h^{-1}\circ f$ for $f\in G^Y$ and $h\in H$.

There are two natural actions of $G\wr H$. In the first $G \wr H$ acts on $X^Y$ (see \cite[Section 4.3]{Cameron});
in this action the wreath product does not generally inherit join- or meet-coherence from $G$ and $H$: for instance
$C_2\wr C_3$ is neither join- nor meet-coherent in this action. For this reason we shall not discuss it
any further here.

The second action of the wreath product 
is the \emph{imprimitive action} on $X\times Y$. If $f:Y\rightarrow G$ and $h\in H$ then
this action is given by
\[ (x,y)fh = (x(yf),yh) \quad\text{for $x\in X,\, y\in Y$\hskip-1pt.} \]
For $y \in Y$, let
\hbox{$B_y= X \times \{y\}$}. Then \hbox{$\{B_y \mid y \in Y\}$} is a
system of imprimitivity for this action, in the sense defined after Proposition~\ref{prop:partitions}.
In general $G\wr H$ does not
inherit meet-coherence from~$G$ and $H$ in the imprimitive action; for instance~$S_3\wr C_2$ is
not meet-coherent.
Join-coherence, however, is inherited in the case that $Y$ is finite, and the proof of this fact
is the
object of this section.

\begin{definition*}
Let $\p$ be a partition of $X\times Y$.
\begin{defnlist}
\item We write $\tilde{\p}$ for the partition of $Y$ defined by
\[
y_1\partequiv_{\tilde{\p}} y_2 \Longleftrightarrow (x_1,y_1)\partequiv_\p (x_2,y_2) \textrm{\ for some\ } x_1,x_2\in X.
\]
\item For $y \in Y$ we write $\p_{[y]}$ for the partition of $X$ defined by
\[
x_1 \partequiv_{\p_{[y]}} x_2 \Longleftrightarrow (x_1,y)\partequiv_{\p} (x_2,y).
\]
\end{defnlist}
\end{definition*}

The following lemma provides a characterization of the orbit partitions of the elements of a wreath product
in its imprimitive action.
For the application to Theorem~\ref{thmmain:products}(b) we only need the case where $Y$ is finite,
which permits some simplifications to the statement and proof.

\begin{lemma}\label{lemma:WreathCrit}
Let $\p$ be a partition of $X\times Y$. Then $\p$ is the orbit partition of an element of $G\wr H$
if and only if the following conditions hold.
\begin{numlist}
\item There exists an element $h\in H$ with orbit partition $\tilde{\p}$ on $Y$.
\item For every $y\in Y$, there exists an element $g \in G$ with orbit partition~$\p_{[y]}$ on~$X$.
\item If $y$ lies in an infinite part of $\tilde{\p}$ then $\p_{[y]}$ is the partition of $X$ into singleton sets.
\item Whenever elements $y$, $z\in Y$ lie in the same part of $\tilde{\p}$, there exists $c\in G$ such that
\[
(x,y)\partequiv_{\p} (xc,z) \textrm{\ for all\ } x \in X.
\]
\end{numlist}
\end{lemma}
\begin{proof}
Let $f\in G^Y$, let $h\in H$, and let $k=fh\in G\wr H$. It is clear that if $\p$
is the orbit partition of $k$ on $X\times Y$, then $\tilde{\p}$ is the orbit partition of $h$ on~$Y$,
and so (1) 
is necessary.
Let $t$ be a positive integer. A simple calculation
shows that $k^t = f_th^t$ where $f_t \in G^Y$ is defined by
\[
f_t: y\mapsto \bigl( yf \bigr) \bigl( (yh)f \bigr) \cdots \bigl( (yh^{t-1}) f \bigr).
\]
Hence
\[
(x,y)k^t = (x (yf_t),yh^t)
\]
for all $x\in X$, and it easily follows that (4)
is necessary. 

Suppose that $y\in Y$ lies in a finite $h$-orbit of size $m$.
In this case it is not hard to see that~$yf_m$ has the orbit partition $\p_{[y]}$ on $X$, since $yf_{am}=(yf_m)^a$ for all $a\in \Z$. On the other hand, if $y$ lies in an infinite orbit of $h$ then
$\p_{[y]}$ is the partition of $X$ into singleton parts, which is the orbit partition of the identity of $G$. So
(2) and (3) are necessary. 

Now suppose that the stated conditions hold. We shall construct a permutation $k \in G \wr H$ such
that $\pi(k) = \mathcal{P}$. By (1) there
exists an element $h$ of $H$ whose orbit partition is $\tilde{\p}$. Let $\{y_i\mid i\in I\}$ be a set of
representatives for the orbits of $h$, where $I$ is a suitable indexing set. Let $m_i \in \N \cup \{\infty\}$
be the size of the orbit containing $y_i$.
By (4) there exists $c_{(i,t)}\in G$ for $i \in I$ such that 
\begin{equation}
\tag{$\star$}
(x,y_ih^t)\partequiv_{\p} (xc_{(i,t)},y_ih^{t+1})
\end{equation}
for all $x\in X$ and $t \in \Z$.

Suppose that $m_i < \infty$. Then the element
\[
b_i = c_{(i,0)}c_{(i,1)}\, \cdots\, c_{(i,m_i-1)}
\]
of $G$ stabilizes the partition $\p_{[y_i]}$ of $X$. It
is possible that $\pi(b_i)$ is a strict refinement of $\p_{[y_i]}$.
However, by condition (2), there exists $g_i \in G$
such that $\pi(g_i) = \p_{[y_i]}$; if we replace $c_{(i,0)}$ with $g_ib_i^{-1}c_{(i,0)}$,
we then have $b_i = g_i$ and so $\pi(b_i) = \p_{[y_i]}$. We may therefore suppose
that the $c_{(i,t)}$ have been chosen so that $\pi(b_i) = \p_{[y_i]}$ for all $i$
such that $m_i < \infty$.

Let $f \in G^Y$ be defined by $(y_i h^t)f = c_{(i,t)}$ for each $i \in I$, where
$t \in \{0, \ldots, m_i-1 \}$ if $m_i < \infty$ and $t \in \Z$ otherwise.
Let $k = fh \in G \wr H$. We shall show
that \hbox{$\pi(k) = \p$}.

We define
\[ b_{(i,t)} = \begin{cases} c_{(i,0)} \cdots c_{(i,t-1)} & \text{if $t \ge 0$} \\
                             c_{(i,-1)}^{-1} \cdots c_{(i,t)}^{-1} & \text{if $t <0$}
\end{cases} \]
and note that
\[
(x,y_i)k^t = (xb_{(i,t)}, y_ih^t)
\]
for all $t \in \Z$. In particular, if $m_i < \infty$, then since $b_{(i,m_i)} = g_i$, we have
$(x,y_i)k^{m_i} = (xg_i, y_i)$. Observe also that, by $(\star)$, we have
\[ (x,y_i) \equiv_\p (xb_{(i,t)},y_ih^t) \]
for all $t \in \Z$.


Suppose that $(x_1,y_ih^s) \partequiv_{\p} (x_2,y_ih^t)$. Then by the observation just made
we have
\begin{align*}
(x_1,y_ih^s)&\partequiv_{\p} (x_1b_{(i,s)}^{-1}, y_i)  \\ 
(x_2,y_ih^t)&\partequiv_{\p} (x_2b_{(i,t)}^{-1}, y_i).
\end{align*}
Hence
\[
x_1b_{(i,s)}^{-1} \equiv_{\p_{[y_i]}} x_2b_{(i,t)}^{-1}.
\]
If $m_i < \infty$, then since $\pi(b_{(i,m_i)}) = \pi(g_i) = \p_{[y_i]}$, the elements
$(x_1b_{(i,s)}^{-1},y_i)$ and 
 $(x_2b_{(i,t)}^{-1},y_i)$ lie in the same orbit of $k^{m_i}$.
  Therefore
%
%
%
\begin{equation}
\tag{$\dagger$}
(x_1,y_ih^s)k^{-s}=(x_1b_{(i,s)}^{-1},y_i) \partequiv_{\pi(k)} (x_2b_{(i,t)}^{-1},y_i)=(x_2,y_ih^t)k^{-t}.
\end{equation}
On the other hand, if $m_i = \infty$, then by (3), $\p_{[y_i]}$ is the partition of $X$ into singleton sets
and so $x_1b_{(i,s)}^{-1} = x_2b_{(i,t)}^{-1}$. Therefore ($\dagger$) also holds in this case.
It follows that
\[ (x_1,y_ih^s)\partequiv_{\pi(k)}(x_2,y_ih^t),\]
and so $\p\preccurlyeq \pi(k)$.



The argument of the previous paragraph  in reverse, again using $(\dagger)$,
implies that $\pi(k)\preccurlyeq\p$, and so we have equality as required.
\end{proof}

With Lemma \ref{lemma:WreathCrit}, we are now in a position to prove the following result, which
combined with Proposition~\ref{prop:imprimitive} completes the proof of Theorem~\ref{thmmain:products}(b).

\begin{proposition}\label{prop:WPjoinc}
Suppose that $G$ and $H$ are join-coherent
permutation groups on $X$ and $Y$ respectively, and that $Y$ is finite.
Then $G \wr H$ is join-coherent in its imprimitive
action on~$X\times Y$.
\end{proposition}
\begin{proof}
Let $f_1h_1$ and $f_2h_2$ be elements of $G\wr H$, and let $K$ be the subgroup they generate.
Let $\p$ be the partition of $X\times Y$ into the orbits of $K$.
It suffices to show that the conditions stated in
Lemma~\ref{lemma:WreathCrit} are satisfied by $\p$. Notice that condition~(3)
 is satisfied vacuously, since~$Y$ is
supposed to be finite.

Since $H$ is join-coherent, it has an element $h$ whose orbit partition $\s$ on~$Y$ is the same as that of the subgroup
$\langle h_1,h_2\rangle$. It is easy to see that $\tilde{\p}=\s$, and so (1) is satisfied. 

For $y\in Y$ let $B_y = X \times \{y\}$;
as noted at the start of this section,
\hbox{$\{B_y\mid y\in Y\}$} is a system of imprimitivity for $G\wr H$ on $X\times Y$.
We write~$R_y$ for the set-wise stabilizer of $B_y$ in $G\wr H$.
In the action of $R_y$ on~$X$ inherited from the action of $R_y$ on $B_y$,
an element $fh \in R_y$ acts on $X$ as~$yf \in G$. Thus if
$\theta : R_y \rightarrow G$ is the homomorphism defined by $(fh)\theta = yf$,
then $(K \cap R_y) \theta$ has orbit partition $\p_{[y]}$.
Now $K\cap R_y$ is finitely generated, since it has finite index in $K$.
Since~$G$ is join-coherent, it follows that
there exists an element $c_y$ of $G$ whose orbit partition is $\p_{[y]}$. This gives us (2). 

Finally suppose that $y$ and $z$ lie in the same part of $\tilde{\p}$.
Then there exist $f\in G^Y$ and $h\in H$ such that $yh=z$, and such that $fh\in K$. Let $c=yf$. We see that $(x,y)\partequiv_{\p} (xc,z)$ for all $x \in X$, and so (4) is satisfied.
This completes the proof. 
\end{proof}

\section{Centralizers}\label{section:Centralizers}

The first part of this paper ends with the proof of Theorem \ref{thmmain:Centralizers}.
The structure of centralizers in symmetric groups is well known.

\begin{lemma}\label{lemma:SymCent}
Let $\Omega$ be a set, let $G=\mathrm{Sym}(\Omega)$,
and let $g\in G$. For $k\in\N\cup\{\infty\}$, let $\Delta_k$ be the set of orbits of $g$ of size $k$. Then
as permutation groups, we have
\[
\Cent_G(g) =\prod_{k\in\N\cup\{\infty\}} \big(C_k\wr \Sym(\Delta_k)\big),
\]
where $C_\infty$ is understood to be $\Z$, the wreath products take the imprimitive action,
and the factors in the direct product
act on disjoint sets. 
\end{lemma}

With notation as in Lemma~\ref{lemma:SymCent},
suppose that $g \in G$ 
has finitely many orbits of all sizes $k \ge 2$,
including $k = \infty$.
Then the join-coherence of $\Cent_G(g)$ follows immediately
from Lemma~\ref{lemma:SymCent} using
 Propositions~\ref{prop:DPc}(1) and~\ref{prop:WPjoinc}. This is sufficient to establish Theorem~\ref{thmmain:Centralizers} so far as join-coherence is concerned. However it tells us nothing about meet-coherence, for which we require an account of orbit partitions in the centralizers of semiregular permutations.


\begin{lemma}\label{lemma:CentCrit}
Let $\Omega$ be a set and let $G=\Sym(\Omega)$. Let $\p$ be a partition of $\Omega$, and let $g\in G$ be an element such that $\langle g\rangle$ acts semiregularly on
$\Omega$. There exists $h\in\Cent_G(g)$ with orbit partition~$\p$ if and only if the following conditions hold:
\begin{numlist}
\item $\p^g = \p$;
\item every part of $\p$ is countable;
\item if $P$ is an infinite part of $\p$, then either $P$ meets only finitely many $g$ orbits, or else the elements of $P$ lie in
distinct $g$ orbits.
\end{numlist}
\end{lemma}

If $m\in\N\cup\{\infty\}$ is the order of the semiregular permutation $g$,
and $\Delta$ is the set of orbits of~$g$, then
by Lemma~\ref{lemma:SymCent}, we have $\Cent_G(g)\cong C_m\wr\Sym(\Delta)$, in its imprimitive action.
Using this, it is possible
to prove Lemma~\ref{lemma:CentCrit} by showing the
equivalence of its conditions with those of Lemma \ref{lemma:WreathCrit}.
However we prefer to give an independent and more illuminating proof
in which we explicitly construct an element of $\Cent_G(g)$ whose orbit partition is $\p$.

\begin{proof}
It is clear that the first two conditions are necessary. To see that the third condition is also necessary, suppose that $\p$ is the orbit partition
of $h\in\Cent_G(g)$. Let $P$ be a part of $\p$, and let $x\in P$.
Since $h$ stabilizes the orbit partition of $g$,
the set $A=\{n\in\Z\mid xh^n \equiv_{\pi(g)} x\}$
is a subgroup of $\Z$. The index
$|\Z:A|$ is equal to the number of $g$-orbits represented in $P$; if this number is infinite then $A=\{0\}$,
since~$\Z$ has no other
subgroup of infinite index.

Now suppose that the stated conditions hold. We shall construct an
element~$h\in\Cent_G(g)$ whose orbit partition is $\p$.
Let $\{P_i\mid i\in I\}$
be a set of orbit representatives for the action of $g$ on the parts of $\p$, where $I$ is a suitable
indexing set. 
Let $S_i=\bigcup_{j \in \Z} P_ig^j$. It is clear that the sets~$S_i$ form a partition $\s$
of
$\Omega$. In fact it is easy to see that
$\s=\p\vee\pi(g)$. 

Since $h$ may be defined separately on the distinct parts of $\s$, we may suppose without
loss of generality that $\s$ is the
trivial partition of $\Omega$ into a single part.

Let $P$ be a part of $\p$, and let $X=\{x_j\mid j\in J\}\subseteq P$ be a set of representatives for the orbits of $g$ on
$\Omega$, where $J$ is a suitable indexing ordinal. If $J$ is infinite then it may be taken
to be the smallest infinite ordinal~$\omega$.
By assumption $P$ is countable, and either $J$ is finite or else $X=P$. Let $t$ be the least positive integer such that
$Pg^t=P$, or $0$ if no such integer exists. It is clear that if $t=0$ then $X=P$.

We have assumed also that $g$ is a semiregular permutation. Let $m$ be the length of a cycle of $g$; here $m$ may be infinite.
For convenience we define $M=\Z/m\Z$ if~$m$ is finite, and $M=\Z$ if $m$ is infinite. We shall allow $g$ to take exponents
from~$M$.
Every element of $\Omega$ has a unique representation as $x_jg^k$
for some $j\in J$ and $k\in M$.

Observe that $x_jg^k$ lies in $P$ if and only if $k$ is a multiple of $t$.
Define $h$ on~$\Omega$ by
\[
x_jg^kh= \left\{\begin{array}{cc} x_{j+1}g^k & \textrm{if $j+1<J$}\\
x_0g^{k+t} & \textrm{if $j+1=J$}.\end{array}\right.
\]
(The second line of the definition, of course, arises only when $J$ is finite.)
Thus $h$ fixes each part of $\p$, and permutes the orbits of $g$.

We now show that $h$ commutes with $g$. It suffices
to show that $x_jg^k$ has the same image under~$gh$ and under $hg$,
for $j \in J$ and $k \in M$. 
Suppose that $j+1<J$; then
\[
(x_jg^k)gh=x_jg^{k+1}h=x_{j+1}g^{k+1}=x_{j+1}g^kg=(x_jg^k)hg.
\]
For the remaining case, we have
\[
(x_{J-1}g^k)gh=x_{J-1}g^{k+1}h=x_0g^{k+1+t}=x_0g^{k+t}g=(x_{J-1}g^k)hg.
\]

It is clear that the points $x_j$ lie in a single orbit $\mathcal{O}$ of $h$ 
and so $X\subseteq \mathcal{O}\subseteq P$. If~$J$ is infinite then clearly $\mathcal{O}=P$, since
$X=P$.
If $J$ is finite then $\mathcal{O}$ contains $x_jg^k$ for each $j \in J$ and each $k \in M$
such that $k$ is a multiple of $t$;
since every element of $P$ is of this form, we have $\mathcal{O} = P$ in this case too.
Now since $g$ and $h$ commute,
every other part of $\p$ is also an orbit of $h$, and so $\pi(h) = \p$ as required.
\end{proof}

We are now in a position to prove Theorem~\ref{thmmain:Centralizers}, which we restate for convenience.
\setcounter{thmmain}{2}
\begin{thmmain}
Let~$\Omega$ be a set, let $G = \Sym(\Omega)$, and let~$g\in G$. 
For \hbox{$k\in\mathbf{N}\cup\{\infty\}$} let~$\n_k$ be the number of orbits of~$g$ of size~$k$. 
\begin{thmlist}
\item $\Cent_{G}(g)$ is meet-coherent. 
\item If~$\n_k$ is finite for all $k\neq 1$, including $k = \infty$,
then $\Cent_G(g)$ is join-coherent. 
\end{thmlist}
\end{thmmain}

\begin{proof}
By Lemma \ref{lemma:SymCent}, we see that $\Cent_{G}(g)$ is the direct product of centralizers of
semiregular permutations. By the two parts of Proposition \ref{prop:DPc},
it will be sufficient to prove the result in the case that $g$ acts semiregularly on~$\Omega$.
We therefore suppose that $g$ is a product of cycles of length $k$, where $k\in\N\cup\{\infty\}$.

Let $\mathcal{P}$ and $\mathcal{Q}$ be partitions in $\Cent_{G}(g)$. By Lemma \ref{lemma:CentCrit} we
have $\mathcal{P}^g=\mathcal{P}$ and $\mathcal{Q}^g=\mathcal{Q}$, from which it follows
easily that $\left(\mathcal{P}\vee\mathcal{Q}\right)^g=\mathcal{P}\vee\mathcal{Q}$ and
$\left(\mathcal{P} \meet \mathcal{Q}\right)^g = \mathcal{P} \meet \mathcal{Q}$.
It is clear that $\mathcal{P} \meet \mathcal{Q}$ also satisfies conditions~(2)
and~(3) of Lemma~\ref{lemma:CentCrit}, and so
$\mathcal{P} \meet \mathcal{Q} \in \pi(\Cent_{G}(g))$. Hence $\pi(\Cent_{G}(g))$ is meet-coherent.

Each part of $\mathcal{P}$ and each part of $\mathcal{Q}$ is countable, and so the parts of
$\mathcal{P}\vee\mathcal{Q}$
are countable. If $k>1$, then by hypothesis $g$ has only finitely many orbits; therefore the only case
in which a part $S$ of $\mathcal{P}\vee\mathcal{Q}$ can meet an infinite number of orbits is when $k=1$,
and clearly in this case~$S$ meets each orbit in at most one point. Hence $\mathcal{P}
\join \mathcal{Q}$ also satisfies conditions (2) and (3) of Lemma~\ref{lemma:CentCrit},
and so $\mathcal{P}\vee\mathcal{Q}\in\pi(\Cent_{G}(g))$. Hence $\pi(\Cent_G(g))$ is join-coherent.
\end{proof}

We have seen that the condition that the orbit multiplicities $\n_k$ are finite for $k>1$
is a sufficient condition for the centralizer in $\Sym(\Omega)$ of
$g$ to be join-coherent. 
Our next proposition shows that this is also a necessary condition.

\begin{proposition}\label{prop:orbitcondition}
Let $\Omega$ be a set and let $G = \Sym(\Omega)$.
Let $g \in G$ be a permutation
which has an infinite number of cycles of length $k$,
for some $k>1$, where $k$ may be infinite. Then $\Cent_G(g)$ is not join-coherent.
\end{proposition}
\begin{proof}
We may assume that $g$ is semiregular, and that all of its orbits have size $k$.
We suppose for simplicity
that~$\Omega$ is countable; the generalization to higher cardinalities is straightforward.
Let~$S$ be a set of representatives for the orbits of $g$. We define $\mathcal{S}$ by
\[
\mathcal{S}= \bigl\{ S^{g^{\raisebox{-1pt}{$\scriptscriptstyle i$}}} \mid i \in M \}
\]
where $M = \{0,\ldots,k-1\}$ if $k < \infty$ and $M = \Z$ if $k=\infty$.

It is easy to see that $\mathcal{S}^g = \mathcal{S}$. By assumption each part of $\mathcal{S}$ is countable,
and it is clear that each part of $\mathcal{S}$ has a single element in common with each orbit of $g$.
Hence $\mathcal{S}$ satisfies the conditions of Lemma~\ref{lemma:CentCrit},
and so $\mathcal{S}\in\pi(\Cent_G(g))$. However $\pi(g) \join \mathcal{S}$ is the trivial partition
of~$\Omega$ into a single part. Since~$k > 1$, this partition does not satisfy
condition (3) of Lemma~\ref{lemma:CentCrit}, and so $\pi(\Cent_G(g))$ is not join-coherent.
\end{proof}


\section{Frobenius groups}
\label{section:Frobenius}
We recall that a \emph{Frobenius group} is a transitive permutation group $G$ on a
finite set $\Omega$, such that each point stabilizer in $G$ is non-trivial,
but the intersection of the stabilizers of distinct points is trivial.
The fixed-point free elements of $G$, together with the identity of~$G$, form a
normal subgroup~$K$,
called the \emph{Frobenius kernel}. The Frobenius kernel acts regularly,
and it is often useful to identify $K$ with $\Omega$ by fixing an element
$\omega \in \Omega$ and mapping $\omega k\in \Omega$ to $k \in G$.
The  stabilizer $H$ of a point $\omega \in \Omega$ is called a \emph{Frobenius
complement}, and acts semiregularly on $\Omega\setminus\{\omega\}$.
Identifying~$\Omega$ with $K$ one finds that any complement $H$ embeds into $\Aut(K)$, and
so~$G$ is isomorphic as a permutation group to $K\rtimes H$, where $K$
has the right regular action on itself. (For 
further results on Frobenius groups
the reader is referred to Theorem 10.3.1 of~\cite{Gorenstein}.)

We give a complete account of join- and meet-coherence in Frobenius groups. As well
as being a natural object of study, these results will be important in
Sections~\ref{section:primitive} and~\ref{section:RNCS}.

\begin{proposition}
A Frobenius group is meet-coherent if and only if it is dihedral of prime degree.
\end{proposition}
\begin{proof}
Let $D$ be the dihedral group of prime degree $p$ acting on a $p$-gon $\Pi$. For any
vertices~$\alpha$ and $\beta$
of $\Pi$, there is a unique reflection mapping $\alpha$ to $\beta$. It follows that if $\mathcal{P}$
and $\mathcal{Q}$ are the orbit partitions
of distinct reflections then $\mathcal{P}\wedge\mathcal{Q}$ is the discrete partition. Since the
only non-identity elements of $D$ are
reflections and full cycles, we see that $D$ is meet-coherent.

For the converse, let $G$ be a meet-coherent Frobenius group with Frobenius kernel $K$,
acting on a set $\Omega$.
If $g$ and $h$ lie
in different point stabilizers, then since $\pi(g) \meet \pi(h)$
has two singleton parts, we see that $\pi(g) \meet \pi(h)$ is the orbit
partition of the identity. It follows that
if $\mathcal{O}_g$ is an
orbit of $g$, and $\mathcal{O}_h$ is an orbit of $h$, then
$|\mathcal{O}_g \cap \mathcal{O}_h| \le 1$.

Suppose that $|G : K| > 2$. Let $\alpha$ and $\beta$
be distinct points in $\Omega$. Since the point stabilizers of $\alpha$ and $\beta$ meet every coset
of $K$, and since there are at least three such cosets, we may choose elements $g \in \Stab_G(\alpha)$ and
$h\in \Stab_G(\beta)$, such that $gh \not\in K$.
Then $gh$ fixes a point $\gamma \in \Omega$, distinct from $\alpha$ and $\beta$.
Now clearly $\{\gamma, \gamma g\}$
is a subset of a $g$-orbit $\mathcal{O}_g$, and $\{ \gamma g, \gamma gh \}$ is
a subset of an $h$-orbit $\mathcal{O}_h$. But since $\gamma gh=\gamma$,
we have found two points in $\mathcal{O}_h \cap \mathcal{O}_g$.
This contradicts the observation above that $|\mathcal{O}_g \cap \mathcal{O}_h| \le 1$.

Therefore $|G : K| = 2$, and if $t\in G\setminus K$, then $t$ acts on $K$ as a fixed-point free
automorphism of order $2$. It is well-known that this implies that $K$ is an abelian group of odd order, and that~$t$ acts by inversion (see \cite[Ch.\ 10, Theorem 1.4]{Gorenstein}).

Suppose that $K$ has a non-trivial proper subgroup $L$.
Since $L$ has odd order, we have
\hbox{$[K : L] \ge 3$}.
If $k\in K$ then the orbits of $k$ are the cosets of $K$,
whereas $t$ preserves $K$ and acts non-trivially on the
cosets $K/L$. It easily follows that $\pi(k) \meet \pi(t)$ is not
the orbit partition of any permutation in $G$; this contradicts the assumption that $G$ is meet-coherent.
Therefore~$K$ is cyclic of prime order $p$, and $G$ is a dihedral group of order~$2p$.
\end{proof}


Since a transitive join-coherent permutation group $G$ of finite degree contains a
full cycle, it is clear that if $G$ is a Frobenius group then its kernel is cyclic.
For this reason, we provide a description of Frobenius groups with a cyclic kernel.
\begin{lemma}\label{lemma:CycFrob}
Let $n\in \N$, and let $H$ be a non-trivial finite group. There exists a Frobenius
group with cyclic kernel
$K \cong C_n$ and complement $H$ if and only if $H \cong C_r$, where
$r$ divides $p-1$ for each prime divisor $p$ of $n$.
\end{lemma}
\begin{proof}
Suppose that there exists such a Frobenius group.
As mentioned at the start of this section, we may consider $H$ as a group of automorphisms of
$K$. Note
that any non-identity element of $H$ acts without fixed points on the non-identity
elements of $K$.
It follows that $n$ is odd, since an even cyclic group has a unique element of order
$2$.
Let $p$ be a prime divisor of~$n$ and
let $L \cong C_p$ be the unique subgroup of $K$ of order $p$.
Then $H$ acts faithfully on $L$ and since $\Aut(L) \cong C_{p-1}$,
it follows that
$H$ is cyclic of order dividing $p-1$. Hence the order
of $H$ divides $p-1$ for each prime divisor $p$ of $n$.

Conversely, let $n=p_1^{a_1}\dots p_m^{a_m}$, let $K = \left< x \right> \cong C_n$,
and let $r$ divide $p_i-1$ for all $i$. The
Chinese Remainder Theorem allows us to choose $d\in \N$ such that $d$ has
multiplicative
order $r$ modulo~$p_i^{a_i}$ for all $i$. The map $h:x\mapsto x^d$ is an
automorphism of $K$ of order $r$, and it is easy to check that $H=\langle h\rangle$
is a Frobenius complement for
$K$.
\end{proof}

We are now in a position to prove the following.
\begin{proposition}\label{prop:Frobjoin}
A Frobenius group is join-coherent if and only if it has prime degree.
\end{proposition}
\begin{proof}
Let $G$ be a Frobenius group with kernel $K$ and complement $H$.
We identify the set on which $G$ acts with $K$.

Suppose that $G$ has prime degree $p$, and so $K \cong C_p$.
Let $X$ be a subgroup of $G$. If $K \le X$ then
$X$ is transitive, and its orbit partition is that of a generator of $K$. Otherwise
$X\cap K$ is trivial,
and since~$G/ K$
is abelian by Lemma~\ref{lemma:CycFrob}, it follows that~$X$ is also abelian.
If~$x$ is a non-identity element of $X$, then $x$ fixes a unique element $a$ of~$K$;
since~$X$
centralizes~$x$ we have $X\le\Stab_G(a)$.
By
 Proposition~\ref{prop:partitions}(3)
 the point stabilizers in $G$ act join-coherently
on $K$, and so there exists $h \in \Stab_G(a)$ such that the orbit
partition of $X$ is $\pi(h)$. Therefore $G$ itself is join-coherent, as required.

For the converse implication,
we observe that the non-identity elements of $K$ are precisely the elements of $G$
whose
orbit partition has no singleton parts. Hence if $G$ is join-coherent then $\{
\pi(k) \mid k \in K \}$
is closed under taking joins. Since the action of $G$ on $K$ is regular,
it follows from Proposition \ref{prop:regularjoin} that $K$ is cyclic.

Suppose for a contradiction that $|K|$ is composite. Then $K$ has a characteristic
non-trivial proper cyclic subgroup
$\langle k\rangle$. The orbit partition $\pi(k)$ 
is simply the
partition of~$K$ into the cosets of
$\langle k\rangle$. Let $h$ be a non-identity element of $H$. By Lemma
\ref{lemma:CycFrob} we see that $|H|$ is coprime with~$|K|$,
and hence with $|\langle k\rangle|$. Since $H$ acts semiregularly on $K\setminus
\{1\}$,
no proper coset of
 $\langle k\rangle$ in~$K$ can be a union of orbits of $h$.
Now consider the partition $\pi(k)\vee\pi(h)$. It has one part equal to
$\langle k\rangle$ itself, and every other part is a
union of two or more cosets of $\langle k\rangle$. But
this partition cannot be in $\pi(G)$, since it has parts of different sizes, but no
singleton parts. Hence~$G$
is not join-coherent. 
\end{proof}

\section{Join-coherence in linear groups}\label{section:linear}

In this section we let $V$ be a vector space of dimension $d$ over a field $K$.
We shall write~$\L$ for the set of lines in $V$.
Let $G$ be a group in the range $\SL(V)\le G\le\GL(V)$.
Then $G$ has a natural action on the non-zero points of $V$,
and the quotient $G/Z(G)$ acts on $\L$.
Since the lines of $V$ form a system of imprimitivity for the
action of $G$, we see by
Proposition~\ref{prop:imprimitive}(1) that
if $G$ acts join-coherently then so does $G/Z(G)$.
The main results
of this section, Proposition~\ref{prop:PGammaLd} and Proposition~\ref{prop:Vjoin},
show that if $d > 1$ then these actions are almost never join-coherent.
These results are used in our classification of
join-coherent primitive groups in Section~\ref{section:primitive}.

When $d=1$ we see that $G$ is a cyclic group acting semiregularly, and so
the action is join-coherent; of course the group $G/Z(G)$ is trivial in
this case. When $d > 1$ we shall see that it is possible to reduce to the case when
$V$ is a $2$-dimensional space; in this case
we identify
$\L$ with the set $K\cup\{\infty\}$, by identifying the line
through $(a, b) \in K^2$ with~$b/a$ when $a\neq 0$, and with~$\infty$ when $a=0$. The group $\PGL_2(K)$
acting on $\L$ may then be identified with the group of fractional linear
transformations 
\[
\alpha \mapsto \frac{a\alpha+b}{c\alpha+d},\quad a,b,c,d\in K,\ ad-bc\neq 0.
\]

The following fact is well known.

\begin{lemma}
The action of the group $\PGL_2(K)$ on $\L$ is sharply $3$-transitive.
\end{lemma}


We shall write $K_0$ for the characteristic subfield of $K$, and $\L_0$
for the subset
$K_0\cup\{\infty\}$ of $\L$.
Our next proposition establishes that $\PGL_2(K)$ and $\PSL_2(K)$ are
join-coherent only in a few small cases.

\begin{lemma}\label{lemma:PGL2}
Let $K$ be a field, and let $G$ be in the range $\SL_2(K)\le G\le \GL_2(K)$.
Then $G/Z(G)$ is join-coherent in its action on $\L$ if and only if $G = \GL_2(\F_2)$
or $G = \GL_2(\F_3)$.
\end{lemma}

\begin{proof}
Suppose first of all that $G$ contains elements of determinant $-1$. (This is always
the case if $\characteristic K = 2$.)
Then $\PGL_2(K)$ has a subgroup isomorphic to $S_3$, generated by
elements $g$ and $h$ which act on $\L$ as
$g : \alpha \mapsto 1/\alpha$ and $h : \alpha \mapsto 1- 1/\alpha$. 
The parts of $\pi(g)\vee\pi(h)$ have the form
\[
\{\alpha, 1/\alpha, 1- \alpha, 1-1/\alpha, \alpha/(\alpha-1), (\alpha-1)/\alpha \}
\]
for $\alpha \in K$.
One part is $\{0,1,\infty\}$.
 If $\characteristic K = 3$ then $\{-1\}$ is a singleton part, and
if $\characteristic K \ge 5$ then $\{-1,2,1/2\}$ is a part of size $3$.
If $K$ contains primitive cube roots of $1$ then they form a part of size $2$.
All other parts have size~$6$.
If $k\in\PGL_2(K)$ has orbit partition $\pi(g) \vee \pi(h)$
then~$k^3$ has at least $3$ fixed points; since $\PGL_2(K)$ is sharply
$3$-transitive, it follows that \hbox{$k^3 = 1$}. Therefore $|K| \le 3$.
The well-known isomorphisms
$\PGL_2(\F_2) \cong S_3$ and $\PGL_2(\F_3) \cong S_4$
now give the join-coherent groups appearing in the lemma.


To deal with the remaining case
it will be useful to observe that if $k \in \PGL_2(K)$ fixes
$\L_0$ set-wise
then there exist $x,y,z \in \L_0$ whose images under~$k$ are
$0,1,\infty$, respectively; since $\PGL_2(K)$ is sharply $3$-transitive
it follows that~$k$ is the map
\[ 
\alpha \longmapsto \frac{(y-z)(\alpha-x)}{(y-x)(\alpha-z)}, \]
and so $k \in \PGL_2(K_0)$.

Suppose that $\characteristic K > 0$ and that $G$ has no elements
of determinant $-1$. It is clear that $G/Z(G)$ acts transitively
on $\L_0$. Suppose that $t \in G$ 
has a single orbit on $\L_0$.
Then $t$ lies in the subgroup of $\GL_2(\F_p)$ generated by a Singer
element $s \in \GL_2(\F_p)$ of order $p^2-1$.
The determinant of $s$ is a generator of $\F_p^\times$ and hence $(\det s)^{(p-1)/2} = -1$.
By assumption $G$ has no elements of determinant $-1$, and so if $p-1 = 2^a c$
where $c$ is odd, then $t \in \left< s^{2^a} \right>$.
However, it is clear that $s$ acts as a $(p+1)$-cycle on $\L_0$, and so~$s^2$
has two orbits on $\L_0$. Hence no such element $t$ can exist,
and so $G$ is not join-coherent in its action on $\L$.

%
Finally suppose that $\characteristic K = 0$.
It is easy to show that $G$ contains elements $g$ and $h$ such that
$\alpha g = 4\alpha$, and $\alpha h = 9\alpha$ for all $\alpha \in K$.
(Here~$4$ and~$9$ may be replaced with any two squares that generate
a non-cyclic subgroup of $\Q^\times$.)
Suppose that $t \in G$ is such that $\pi(t) = \pi(g) \join \pi(h)$.
Since $t$ has both $0$ and $\infty$ as fixed points, it easily follows that
there exists $x \in \Q$ such
that $\alpha t = x\alpha$ for all $\alpha \in K$. The orbit
of $\left<g,h\right>$ on~$\Q \cup \{\infty\}$
containing $1$ is $\{4^i 9^j : i,j \in \Z\}$,
whereas the orbit of $\left< t \right>$ on $\Q \cup \{\infty\}$
containing $1$ is $\{ x^i : i \in \Z \}$.
It is clear that these sets are not equal for any choice of $x \in \Q$.
Hence $G/Z(G)$ is not join-coherent in its action on~$\L$.
\end{proof}

It is worth noting that $\PSL_2(\Q)$ does
contain finitely generated subgroups that are transitive on $\Q \cup \{\infty\}$.
For example, one such subgroup is generated by the maps
$g: \alpha \mapsto \alpha+1$ and $h:\alpha \mapsto \alpha/(\alpha+1)$
used to define the Calkin--Wilf tree of rational numbers
\cite{CalkinWilf}. It can be shown, however, that no
element $k \in \PGL_2(\Q)$ acts transitively
on $\Q \cup \{\infty\}$; this gives an alternative
way to conclude the proof of Lemma~\ref{lemma:PGL2}.

Lemma \ref{lemma:PGL2} is the basis for the following more general
statement.

\begin{proposition}\label{prop:PGLd}
Let $V$ be a vector space of dimension $d$ over a field $K$, where
$d \ge 2$. Let~$G$ be in the range 
$\SL(V)\le G\le \GL(V)$. Then the action of $G/Z(G)$ on the lines of~$V$
is not join-coherent unless
$d=2$ and $|K|\le 3$.
\end{proposition}

\begin{proof} Let $W$ be a 
subspace of $V$, and let~$G_W$ be the
set-stabilizer of $W$ in~$G$.
By Proposition~\ref{prop:partitions}(2), if
$G$ is join-coherent in its projective action on the lines of $V$,
then $G_W$ is join-coherent on the lines of~$W$.

Take $W$ to be a $2$-dimensional subspace of $V$. By
Lemma~\ref{lemma:PGL2}, applied to $G_W$,
we see that~$G$ cannot be join-coherent unless $|K|\le 3$.
Suppose that $|K| \le 3$ and $d \ge 3$. Take $W$ to
be a $3$-dimensional subspace of $V$.
The possibilities for $G_W / Z(G_W)$ are $\PGL_3(\F_2)$ and $\PSL_3(\F_3)$
and $\PGL_3(\F_3)$. A straightforward computation, using the software mentioned
in the introduction, shows that none
of these groups is join-coherent, and so if
 $d\ge 3$ then~$G/Z(G)$ is not join-coherent.
\end{proof}

Let $\Phi$ be a non-trivial group of automorphisms of the field $K$, and
let $G$ be a group such that
$\SL(V)\le G\le \GL(V)$.
Then $G\cdot \Phi$ acts on the
non-zero points of $V$, and
$(G\cdot \Phi)/Z(G)$ acts on the set $\L$ of lines.

\begin{lemma}\label{lemma:PGammaL2}
Let $\Phi$ be a non-trivial group of automorphisms of a finite field~$K$.
Let $G$ be in the range
$\SL_2(K)\le G\le \GL_2(K)$.
Then $(G\cdot\Phi)/Z(G)$ is join-coherent on $\L$ if and only if $K=\F_4$, $G=\SL_2(\F_4)$
or $G = \GL_2(\F_4)$,
and $\Phi=\Gal(\F_4:\F_2)$.
\end{lemma}
\begin{proof}
Let $|K|=p^r$ where $p$ is prime. Since $K$ admits non-trivial
automorphisms, we must have $r>1$.
Let $H$ be the stabilizer in $G \cdot \Phi$ of a distinguished line $\ell \in \L$.
Since the action of $\PSL_2(K)$ on $\L$ is $2$-transitive,
the action of $H$ on $\L \backslash \{ \ell \}$ is transitive.
If $G \cdot \Phi$ is
join-coherent on~$\L$,
then by Proposition \ref{prop:partitions}(3), so is~$H$.
It follows that
$H$ must contain an element of order $p^r$.
Let $r=p^am$ where $p$ does not divide~$m$. Let $g\in G \cdot \Phi$ be a $p$-element.
Since the full automorphism
group of $K$ has order $r$, we see that $g^{p^a} \in G$. 
But a
non-trivial unipotent element of $\GL_2(K)$ has order
$p$, and so $g^{p^{a+1}} = 1$. Hence the order of $g$ is at
most $p^{a+1}$. However, $p^{a+1}$ is strictly less than $p^r$,
except in the case when $a=1$, $r=2$ and $p=2$, and so $p^r=4$.

When $K=\F_4$ we observe that since every element in $\F_4$ is a square,
$\PGL_2(\F_4) = \PSL_2(\F_4)$. Moreover,
$\PGL_2(\F_4) \cdot \Gal(\F_4 : \F_2)$
is isomorphic to the symmetric
group $S_5$, in its standard action on $5$ points, and is
therefore join-coherent.
\end{proof}

The principal difficulty in extending Lemma~\ref{lemma:PGammaL2} to general fields
comes from simple transcendental extensions of $\F_p$ for small primes $p$.
Let $K = \F_p(x)$, where $x$ is a transcendental element.
We shall represent elements of $K\cup \{\infty\}$
as rational quotients $P(x)/Q(x)$, where not both $P(x)$ and $Q(x)$ are zero,
taking the quotients in which $Q(x) = 0$ to represent $\infty$.
Recalling our earlier identification of $K\cup\{\infty\}$ with the projective line $\L$, this gives
a convenient representation for the points of $\L$.
In this representation, the fractional linear transformation
 $\alpha \mapsto (A\alpha + B)/(C\alpha + D)$ in $\PGL_2(K)$
acts by
\[ \frac{P(x)}{Q(x)} \mapsto \frac{A P(x) + B Q(x)}{CP(x) + DQ(x)}. \]
It is straightforward to check that this is a well-defined action of
$\PGL_2(K)$ on $\L$.



\begin{lemma}\label{lemma:PGLPGL}
Let $p$ be prime and let $K = \F_p(x)$, where $x$ is a transcendental element.
Let $\Phi = \Gal(K : \F_p)$ and let $H = \PGL_2(K)$.
There is an action of $H \cdot \Phi$ on $\L$ defined by
\[ h\phi : \frac{P(x)}{Q(x)} \rightarrow \left( \frac{P(x\phi)}{Q(x\phi)} \right) h.
\]
In this action, $H\cdot \Phi$ acts regularly on the orbit containing $x^{p+1} + x^p$.
\end{lemma}

\begin{proof}
The remarks made immediately before the lemma show that the action is well-defined.
Let $P(x) = x^{p+1} + x^p$. It suffices to show that if $P(x)h = P(x\phi)$ for
\hbox{$h\in H$} and $\phi\in\Gal(K :\F_p)$, then $h$ and $\phi$ are the identities
in their respective groups.
It is clear that $\phi$ is determined by its effect on $x$, and it is well known that
\[ x \phi = \frac{ax+b}{cx+d} \]
for some $a,b,c,d \in \F_p$ with $ad-bc\not=0$. (In fact $\Gal(K : \F_p) \cong
\PSL_2(\F_p)$, and so we have two distinct projective linear groups, acting in
different ways on~$\L$.)
Suppose that
\[
P(x\phi)= \frac{ (ax+b)^{p+1} }{ (cx+d)^{p+1} } + \frac{ (ax+b)^p }{ (cx+d)^p }
= \frac{ A(x^{p+1} + x^p) + B }{ C(x^{p+1} + x^p) + D } = Ph.
\]
Then using the fact that $(rx+s)^p = rx^p + s$ for any $r$, $s\in K_0$, we have
\[
\begin{split}
(ax^p +b)\bigl( (a+c)x + {}&{} (b+d) \bigr)(Cx^{p+1} + Cx^p + D)\\
&= (cx^p +d)(cx+d)(Ax^{p+1} + Ax^p + B).
\end{split}
\]
By comparing the coefficients of $x$ on both sides of this equation,
starting with the constant and linear terms, it is now easy to show that $h$ and $\phi$ are the identity maps.
%
%
\end{proof}

We are now ready to prove our main result on the action of extended linear groups
on lines.

\begin{proposition}\label{prop:PGammaLd}
Let $\Phi$ be a non-trivial group of automorphisms of a field~$K$. Let~$V$
be a $d$-dimensional space over~$K$, and
let $G$ be in the range $\SL(V)\le G\le \GL(V)$. If \hbox{$(G\cdot\Phi)/Z(G)$} is
join-coherent on the lines of $V$, then
$K=\F_4$, $G = \SL_2(\F_4)$ or $G=\GL_2(\F_4)$, and \hbox{$\Phi=\Gal(\F_4:\F_2)$}.
\end{proposition}

\begin{proof}
Let $W_0\subseteq V$ be a $2$-dimensional vector space over $K_0$, and define
\hbox{$W=W_0\otimes_{K_0} K$}. Then~$\Phi$ stabilizes $W$ as a set.
By Proposition \ref{prop:partitions}(2),
if $(G\cdot\Phi)/Z(G)$ is
join-coherent on the lines of $V$, then the set stabilizer of $W$ is
join-coherent on the lines of $W$.
Therefore, provided $K \not= \F_4$, it is sufficient to prove the theorem in the case $d=2$.
By a similar argument, it is sufficient in the case $K=\F_4$ to show that the
groups \mbox{$\PGL_3(4)\cdot \Gal(\mathbf{F}_4 : \mathbf{F}_2)$} and
\mbox{$\PSL_3(4)\cdot \Gal(\mathbf{F}_4 : \mathbf{F}_2)$} are
not join-coherent. This follows from a straightforward computation.

We shall therefore assume that $V$ is $2$-dimensional.
We need the following observation: if $E$ is a subfield of $K$
and $H_E$ is the set-stabilizer of the set $\L_E$ of lines in $\L$ contained in
$W_0 \otimes_{K_0} E$, then by Proposition~\ref{prop:partitions}(2),
the action of $H_E$ on $\L_E$ is join-coherent.

We first use this observation in the case $E = \mathbf{F}_p$. Then
$\L_E = \L_0$, and since~$\Phi$ acts trivially on $\L_0$, we may
apply Lemma~\ref{lemma:PGL2} to the group $H_E \le \GL_2(\F_p)$ 
to deduce that $|K_0| \le 3$.
Hence $K_0 = \mathbf{F}_p$ where $p \le 3$. If $K$ is algebraic
over~$\F_p$ then $K$ is a finite field, and so, by Lemma~\ref{lemma:PGammaL2},
we have $K = \F_4$ and $G = \SL_2(\F_4)$ or $G = \GL_2(\F_4)$.

Now suppose that $K$ is not an algebraic extension of $\F_p$. In
this case there exists $x \in K$ such that $x$ is transcendental over $\F_p$.
Applying the observation to the purely transcendental extension
$E = \F_p(x)$, we see that $H_E \cdot \Phi$ acts join-coherently on $\L_E$.
%
By Lemma~\ref{lemma:PGLPGL}, the group~$H_E \cdot \Phi$ has a regular orbit in its action
on $K \cup \{\infty\}$.
However~$H_E \cdot \Phi$ is not locally cyclic, and so, by Proposition~\ref{prop:regularjoin},
the action of $H_E \cdot \Phi$ is not join-coherent.
\end{proof}

Finally, we extend the results to the action of~$\GL(V)$ on the non-zero
points of $V$.

\begin{proposition}\label{prop:Vjoin}
Let $V$ be a $d$-dimensional vector space over a field $K$, where $d>1$.
Let~$\Phi$ be a group (possibly trivial) of automorphisms of $K$,
and let $G$ be in the range $\SL(V) \le G\le \GL(V)$. If $G\cdot \Phi$ is
join-coherent in its action on $V\setminus\{0\}$
then $K=\F_2$ and $d=2$.
\end{proposition}

\begin{proof}
The set of punctured lines $\{ \ell \setminus \{0\} \mid \ell \in \L \}$
form a system of imprimitivity for the action of $G$ on $V \setminus \{0\}$.
Hence, by Proposition~\ref{prop:imprimitive}(1),
if $G\cdot \Phi$ is
join-coherent on points, then $(G\cdot\Phi)/Z(G)$ is join-coherent on
lines. It follows from
Propositions \ref{prop:PGLd} and \ref{prop:PGammaLd} that $G\cdot\Phi$ is
one of $\GL_2(\F_2)$, $\GL_2(\F_3)$ or
\mbox{$\GL_2(\F_4)\cdot \Gal(\F_4 : \F_2)$}. A  computation shows that the
only one of these groups which is join-coherent on points
is $\GL_2(\F_2)$.
\end{proof}

\section{Primitive join-coherent groups of finite degree}
\label{section:primitive}

In this section we shall establish Theorem~\ref{thmmain:primitive},
that a primitive permutation group of finite degree is join coherent
if and only if it is a symmetric
group or a subgroup of the Frobenius group $\AGL_1(\F_p)$
in its action on $p$ points.

If $G$ is a primitive join-coherent group of finite degree $n$ then it contains an $n$-cycle.
The following lemma, classifying such groups, is Theorem~3 in~\cite{Jones}.
(The reader is referred to \cite{Jones} for original references.)
We shall write $\PGammaL_d(\F_q)$ for the group
$\PGL_d(\F_q) \cdot \Phi$, where~$\Phi$ is the Galois group of $\F_q$ over
its prime subfield.

\begin{lemma}\label{lemma:primitivefullcycle}
Let $G$ be a primitive permutation group on $n$ points containing an
$n$-cycle. Then one of the following holds:
\begin{numlist}
\item $G$ is $S_n$ or $A_n$;
\item $n=p$ is a prime, and $G\le \AGL_1(\F_p)$;
\item $\PGL_d(\F_q)\le G\le \PGammaL_d(\F_q)$ for $d>1$, where
$n=(q^d-1)/(q-1)$, the action being either on projective
points or on hyperplanes;
\item $G$ is $\PSL_2(\F_{11})$ or $M_{11}$ acting on $11$ points, or $M_{23}$
acting on $23$ points.
\end{numlist}
\end{lemma}

%
%


It is straightforward to prove Theorem~\ref{thmmain:primitive} using
this lemma and the results of Sections~7 and~8.
Certainly $S_n$ is join-coherent, and we have seen that~$A_n$ is not
join-coherent when $n>3$.
A transitive subgroup of $\AGL_1(\F_p)$ is a Frobenius group, and is
therefore join-coherent by Proposition \ref{prop:Frobjoin}.

Suppose that $\PGL_d(\F_q)\le G\le \PGammaL_d(\F_q)$. The actions of $G$
on points and on hyperplanes are dual to one another, and it therefore
suffices to rule out join-coherence for one
of them. By Propositions~\ref{prop:PGLd} and~\ref{prop:PGammaLd}, the only join-coherent examples
in the action on points
are $\PGL_2(\F_2)$, $\PGL_2(\F_3)$ and
$\PGammaL_2(\F_4)$. As we have seen, these are isomorphic as permutation groups to~$S_3$,~$S_4$
and~$S_5$ respectively, in their natural actions.

We have therefore reduced the proof to a small number of low degree
groups, namely
$\PSL_2(\F_{11})$ in its action on $11$ points, and the Mathieu groups $M_{11}$
and $M_{23}$.
Establishing that none of these groups is join-coherent is a routine
computational task.

\section{Groups containing a proper normal cyclic subgroup acting regularly}
\label{section:RNCS}

We have observed that a transitive join-coherent permutation group on a finite set
must contain a full cycle. We end this paper by investigating the situation when this cycle generates
a normal subgroup.

Let $G$ act on $\Omega$, a set of size $n$. Suppose that $K$ is a transitive normal cyclic subgroup
of~$G$ of order $n$. Let $H$ be the stabilizer
of a point $\omega\in\Omega$. Then clearly $G=K\rtimes H$, and by the argument indicated
at the start of Section 7, we
may identify $\Omega$ with $K$ by the bijection sending $\omega k \in \Omega$ to $k \in K$.
The action of $H$ on $\Omega$
then defines an embedding of $H$ into $\Aut(K)$.
Every subgroup of $K$ is characteristic in $K$, and therefore invariant under the action of~$H$.

Suppose that $G$ is join-coherent. If $n=ab$ for coprime $a$, $b$ then $K\cong C_a\times C_b$, and we see from Proposition \ref{prop:DPconverse} that
$G$ factorizes as $G_1\times G_2$, where~$G_1$ is join-coherent on $C_a$, $G_2$ is join-coherent on $C_b$, and the factors~$G_1$ and~$G_2$ have coprime orders.
Therefore, to obtain a complete classification, it suffices to consider the case that $n$ is a prime power.

One trivial possibility is that $G=K$; in this case the action of $G$ is semiregular, and join-coherent by Proposition \ref{prop:regularjoin}. If $K$ is assumed to be a proper subgroup, then we shall
see that the classification divides into two cases: the case that $n=p$ is prime, and the case that $n = p^a$ for $a>1$.

\begin{proposition}\label{prop:RNCSextension}
Let $p$ be prime and let $a>1$. Let $\Gamma(p^a)$ be the extension of the additive group $\Z/p^a\Z$ by the
automorphism $f:x\mapsto rx$, where $r=p^{a-1}+1$. Then $\Gamma(p^a)$ is join-coherent.
\end{proposition}
\begin{proof}
An element $g$ of $\Gamma(p^a)$ may be represented as $x\mapsto r^jx+ i$ for non-negative integers $i<p^a$ and $j<p$. It is clear that $xg = (jp^{a-1}+1)x+i$ and so $xg = x$
if and only if \hbox{$jp^{a-1}x + i =0$}.
Moreover, a straightforward calculation shows that
\[
xg^t-x = tjp^{a-1}x+\frac{t(t-1)}{2}ijp^{a-1}+ti,
\]
and so, if $c \ge 1$, then
\[ xg^{p^c} - x = \begin{cases} p^ci & \text{if $p$ is odd} \\
                           2^c i \bigl( 2^{a-2}(2^c-1) j + 1\bigr) & \text{if $p= 2$.} \end{cases} \]
Since $\Gamma(p^a)$ is a $p$-group, and $xg^t - x = 0$ if and only if the $g$-orbit
containing~$x$ has size dividing~$t$,
the preceding equation
allows us to determine the sizes of the orbit partitions of elements of $\Gamma(p^a)$.
Let~$p^b$ be the highest power of $p$ dividing~$i$ if $i \not=0$,  and let~$b = a$ if $i = 0$.
\begin{defnlist}
\item If $b<a-1$, or if $j=0$, then $xg \equiv x$ mod $p^{b}$ for all $x \in \Z/p\Z$
and each orbit of $g$ has size $p^{a-b}$. Hence the orbits of $g$
are the cosets of $\langle p^b \rangle$ in $\Z/p^a\Z$.
\item If $b \ge a-1$ and $j\neq 0$, then $x$ is a fixed point of $g$ if and only if $p$ divides $jx + k$,
where~$i = p^{a-1}k$. Thus the fixed points of $g$ form a coset of $\langle p \rangle$ in $\Z/p^a \Z$.
The remaining orbits have size $p$ and are cosets of $\langle p^{a-1} \rangle$ in $\Z / p^a \Z$.
\end{defnlist}

\noindent From this description of the orbit partitions
of elements of $\Gamma(p^a)$, it is not hard to
show that~$\pi(\Gamma(p^a))$ is closed under the join operation.
\end{proof}

We remark that when $p$ is odd, the group $\Gamma(p^a)$ is the unique extension of $\Z/p^a\Z$ by an automorphism of order $p$. When $p=2$ there are
three such extensions, provided $a \ge 3$, of which $\Gamma(p^a)$ is the one which is neither dihedral nor quasidihedral.
(There appears to be no widely accepted name for this group.)

\begin{proposition}\label{prop:RNCSconverse}
Let $p$ be prime, let $a>1$, and let $K$ be the additive group $\Z/p^a\Z$.
Let~$H$ be a non-trivial group of automorphisms of $K$.
The group \hbox{$K\rtimes H$} is join-coherent if and only if it is the group $\Gamma(p^a)$ from
Proposition~\ref{prop:RNCSextension}.
\end{proposition}
\begin{proof}
Let $G$ be the full group of affine transformations of $\Z/p^a\Z$. Then $K\rtimes H\le G$.
Let~$L$ be the unique subgroup of $K$ of order $p$.

We describe the elements of $G$ which have an orbit equal to $L$. Suppose that $g:x\mapsto rx+s$ is such an element. Then
it is easy to see that $s=mp^{a-1}$ for some $m$ not divisible by $p$, since the image of $0$ under $g$ is a non-identity element of $L$. Furthermore, since
$g$ has no fixed points in $L$, we see that $r\equiv 1\bmod p$. But these restrictions on $r$ and $s$ imply that the equation $x=rx+s$ has a solution in $\Z/p^a\Z$, except in the
case that $r=1$, when $g\in L$. Hence
$g$ must have a fixed point in $K\setminus L$, except in the case that its orbits are precisely the
cosets of $L$ in $K$.

Let $g$ be a generator of $L$. The orbits of $g$ are the cosets of $L$, and the automorphism group~$H$
clearly stabilizes $L$ set-wise. It follows that
for any $h\in H$, the join $\pi(g)\vee\pi(h)$ has $L$ as a part,
and that every part is a union of cosets of $L$. But such a partition has no singleton part, and so
cannot be in $\pi(G)$ unless each of its parts is a single coset; this implies that if $K\rtimes H$
is join-coherent, then the action of $H$ on the cosets of $L$ is trivial.
It is easy to see that this is the case only if $H= \langle f\rangle$,
where $f$ is as in Proposition~\ref{prop:RNCSextension}.
\end{proof}

We are now in a position to prove Theorem \ref{thmmain:RNCS}.
For convenience we restate the theorem below.
\setcounter{thmmain}{4}
\begin{thmmain} 
Let~$G$ be a permutation group on~$n$ points, containing a normal cyclic subgroup of order~$n$ acting regularly. Let~$n$ have prime factorization~$\prod_ip_i^{a_i}$.
Then~$G$ is join-coherent if and only if there exists for each~$i$ a transitive permutation
group~$G_i$ on~$p_i^{a_i}$ points, such that:
\begin{bulletlist}
\item[$\bullet$] if $a_i>1$ then~$G_i$ is either cyclic or the extension of a cyclic group of
order~$p_i^{a_i}$ by the automorphism
$x\mapsto x^r$ where $r = p_i^{a_i-1}+1$,
\item[$\bullet$] if $a_i=1$ then~$G_i$ is a subgroup of the Frobenius group of order $p(p-1)$,
\item[$\bullet$] the orders of the groups~$G_i$ are mutually coprime,
\item[$\bullet$] $G$ is permutation
isomorphic to the direct product of the groups~$G_i$ in its product action.
\end{bulletlist}
\end{thmmain}

\begin{proof}
Let $n=\prod_ip_i^{a_i}$, and suppose that $G$ is a join-coherent permutation group on $n$ points containing
a regular normal cyclic subgroup $K$ of order $n$.
Then we can regard $G$ as acting on~$K$. Let $K_i$ be the unique subgroup of $K$ of order $p_i^{a_i}$.
Since $K \cong \prod_i K_i$, it follows easily from
Proposition~\ref{prop:DPconverse} that $G\cong \prod_i G_i$,
where $G_i$ is the kernel of $G$ in its action on
the complement~$\prod_{j \not=i} K_j$ of $K_i$. Moreover Proposition~\ref{prop:DPconverse}
implies that the groups $G_i$ and $G_j$ have coprime orders whenever
$i\neq j$, and that $G_i$ acts join-coherently on $K_i$ 
for all $i$. If $a_i>1$, then by Proposition~\ref{prop:RNCSconverse},
either $G_i$ is cyclic of order $p_i^{a_i}$ or $G_i$ is isomorphic to $\Gamma(p_i^{a_i})$,
while if $a_i = 1$ then $G_i$ is a subgroup of the normalizer in $S_{p_i}$ of a
$p_i$-cycle,
and so is a subgroup of the Frobenius group of order $p_i(p_i-1)$.
This completes the proof in one direction.

For the converse, suppose that we have for each $i$ a permutation group $G_i$ on $p_i^{a_i}$ points,
containing a regular normal cyclic subgroup, and such that if $a_i>1$ then
$G_i$ is either cyclic or isomorphic to $\Gamma(p_i^{a_i})$.
If $a_i>1$ then Proposition \ref{prop:RNCSextension} tells us that
$G_i$ is join-coherent. If $a_i=1$ on the other hand,
then $G_i$ is either cyclic or else a Frobenius group of prime degree,
in which case it is join-coherent by Proposition \ref{prop:Frobjoin}.
If the orders of the groups~$G_i$ are coprime, then their direct product is
join-coherent by Proposition \ref{prop:DPjoin}, and this completes the proof.
\end{proof}

We end with a remark on the uniqueness of the decomposition in
Theorem~\ref{thmmain:RNCS}.
Since the groups $K_i$ and $G_i$ appearing in the proof of the theorem are uniquely
determined, it follows that any group which satisfies the hypotheses
of this theorem has a unique decomposition into a direct product of
transitive groups of prime power degrees.


\end{document}